\documentclass[a4paper, 10pt]{article}

\usepackage[english]{babel}
\usepackage[latin1]{inputenc}
\usepackage{graphicx}
\usepackage{amsmath}    
\usepackage{amssymb}
\usepackage{amsmath,amsthm}
\usepackage{cancel}
\usepackage{textcomp}

\usepackage{mathtools}            %
\mathtoolsset{showonlyrefs}       

\newcommand{\RR}{\mathbb{R}}		
		
\newcommand{\ZZ}{\mathbb{Z}}

\newcommand{\eps}{\varepsilon}
\newcommand{\Sigmadue}{{\{0,1\}^{\mathbb{Z}}}}

\newtheorem{theorem}{Theorem}[section]

\newtheorem{lemma}[theorem]{Lemma}
\newtheorem{proposition}[theorem]{Proposition}

\theoremstyle{definition}

\newtheorem{remark}{Remark}[section]

\titlepage
\title{\bf \Large Positive solutions with a complex behavior for superlinear indefinite ODEs on the real line}

\author{%
Vivina Barutello\footnote{Dipartimento di Matematica, Universit\`a degli Studi
di Torino, Via Carlo Alberto, 10,  10123 Torino,
Italy. e-mail: \texttt{vivina.barutello@unito.it}}
\and
Alberto Boscaggin\footnote{Dipartimento di Matematica e Applicazioni, Universit\`a degli Studi
di Milano-Bicocca, Via Cozzi, 53, 20125 Milano,
Italy. e-mail: \texttt{alberto.boscaggin@unimib.it}}
\and
Gianmaria Verzini\footnote{Dipartimento di Matematica, Politecnico di Milano, Piazza
Leonardo da Vinci, 32,  20133 Milano, Italy.
e-mail: \texttt{gianmaria.verzini@polimi.it}}
}
\date{\today}

\begin{document}

\maketitle

\normalsize
\begin{abstract}
We show the existence of infinitely many positive solutions, defined on the real line,
for the nonlinear scalar ODE
\[
\ddot u + (a^+(t) - \mu a^-(t)) u^3 = 0,
\]
where $a$ is a periodic, sign-changing function, and the parameter $\mu>0$ is large. Such solutions
are characterized by the fact of being either small or large in each interval of positivity of $a$.
In this way, we find periodic solutions, having minimal period arbitrarily large, and bounded non-periodic
solutions, exhibiting a complex behavior. The proof is variational, exploiting suitable natural constraints
of Nehari type.
\end{abstract}

\noindent
{\footnotesize \textbf{AMS-Subject Classification}}. {\footnotesize 34B18, 34C25, 34C28, 37J45}

\section{Introduction}

In this paper, we deal with the existence of \emph{positive} bounded solutions, with a complex behavior,
of the nonlinear scalar ODE
\begin{equation}\label{eq_intro}
\ddot u + q(t)u^3 = 0, \qquad t\in\RR,
\end{equation}
where $q(t)$ is a bounded and $T$-periodic function (for some $T > 0$) which changes its sign.
According to a terminology which is now quite standard in this setting (see \cite{HesKat80}), equation
\eqref{eq_intro} is thus \emph{superlinear indefinite}.

When understanding $t$ as a space variable, equation \eqref{eq_intro} can be seen as a toy model
of the elliptic PDE
\begin{equation}\label{eqPDE}
\Delta u + \lambda u + q(x)u^p = 0, \qquad x \in \Omega \subset \mathbb{R}^d,
\end{equation}
with $\lambda \in \mathbb{R}$ and $p > 1$, which in turns arises when searching for steady states of the corresponding
evolutionary parabolic problem (see \cite{Ack09} for a recent survey on the topic). Such kind of equations has a
typical interpretation in the context of population dynamics, with the unknown $u$ playing the role of
density of a species inhabiting the spatially heterogeneous domain $\Omega$%
. Accordingly, the (indefinite) sign of the coefficient
$q$ expresses saturation or autocatalytic behavior of the species $u$, when $q \leq 0$ or $q \geq 0$ respectively.
Classical existence results, obtained both with topological and
 variational methods,
for positive solutions of boundary value problems associated with \eqref{eqPDE}
can be found among others in \cite{AlaTar93,AmaLop98,BerCapNir94,
BerCapNir95,GauHabZan03b,Lop00}.

Wishing to investigate the complexity of the solution set for \eqref{eqPDE},
a typical strategy requires to play with the nodal behavior
of the weight function $q$ and/or - in the PDE case - with the shape of
$\Omega$, so as to regard \eqref{eq_intro}-\eqref{eqPDE}
as singular perturbation problems (see \cite{Dan88,Dan90,GomLop00}). For instance,
in \cite{GauHabZan03,GauHabZan04} Gaudenzi, Habets and Zanolin
dealt with the ODE \eqref{eq_intro} assuming that
$$
q(t) = a_\mu(t) := a^+(t) - \mu a^-(t),
$$
with $\mu$ a real parameter and $a^+$, $a^-$ the positive and the negative part of
a sign-changing function $a$, and they proved the existence of multiple positive solutions
for
\begin{equation}\label{eqmain_intro}
\ddot u + a_\mu(t) u^3 = 0
\end{equation}
when $\mu \gg 1$. More precisely, they showed that the two-point boundary value
problem $u(0) = u(T) = 0$ associated with \eqref{eqmain_intro} has at least $2^n - 1$ positive solutions (for $\mu$ large)
whenever the weight function $a$ has $n$ disjoint intervals of positivity in $[0,T]$ (separated by intervals of negativity).
The number $2^n - 1$ comes from the possibility of prescribing, for a positive solution
of \eqref{eqmain_intro}, the behavior on each interval of positivity of $a$
among two possible ones: either the solution is ``small'' or is ``large''
(notice that the solution small on all the intervals is excluded, since it corresponds to the trivial one).
Such a result, which was originally proved with a shooting technique, has later
been generalized, using variational tools, by Bonheure, Gomes and Habets \cite{BonGomHab05} and by
Gir{\~a}o and Gomes \cite{GirGom09,GirGom09b} to the Dirichlet problem for the corresponding elliptic PDE.
Very recently, a topological approach for the ODE case has also been proposed by Feltrin and
Zanolin \cite{FelZanPP}.

It is the aim of the present paper to show that such kind of results
has a natural analogue when the equation \eqref{eqmain_intro} is considered
on an infinite interval (with $a$ a $T$-periodic function).
That is, we can still produce, for $\mu$ sufficiently large, positive solutions
of \eqref{eqmain_intro} defined on the real line and being either small or large on the intervals of positivity
of the weight function $a$ according to a prescribed rule
(which now involves the behavior of solutions on infinitely many intervals).
More precisely, we prove the
existence of positive $mT$-periodic solutions to \eqref{eqmain_intro} for
any integer $m \geq 1$ (namely,
positive subharmonic solutions)
and, eventually, of bounded non-periodic solutions
with a complex behavior, according to a typical scheme of ``chaotic dynamics''.
Here is a simplified statement of our main existence result (for more details see
Section \ref{sec:mainres}).
\begin{theorem}\label{thmain_intro}
Let $a$ be $L^\infty(\RR)$, $T$-periodic, and such that, for every $i\in\ZZ$,
\[
a(t) \geq 0, \;\text{ a.e. on } I^+_i, \qquad
a(t) \leq 0, \;\text{ a.e. on } I^-_i,
\]
(but not identically zero) where the consecutive closed intervals $I_i^+,I_i^-$ are such that $[0,T]$
is the union of a finite number of them.

For every integer $k \geq 1$, there exists
$\mu^* > 0$ such that, for every $\mu > \mu^*$
and for every double-sequence $\mathcal{L} \in \Sigmadue$ containing strings of zeroes of length
at most $k$,
the equation \eqref{eqmain_intro} has a positive solution $u \in W^{2,\infty}(\mathbb{R})$
such that, for every $i \in \ZZ$,
\begin{equation}\label{gobbe_intro}
u|_{I^+_i} \quad \mbox{ is ``small'' if } \; \;\mathcal{L}_i = 0
\qquad \mbox{ and } \qquad
u|_{I^+_i} \quad \mbox{ is ``large'' if } \; \;\mathcal{L}_i = 1
\end{equation}
(and $u|_{I^-_i}$ is ``small'' for every $i$).

Moreover, such a solution can be chosen to be
periodic whenever the sequence $\mathcal{L}$ is periodic.
\end{theorem}
It is worth mentioning that some results dealing with complex dynamics for ODEs with indefinite weight
have already appeared (see, among others, \cite{BosZan12,CapDamPap02,PapZan02,TerVer00}).
However, \cite{CapDamPap02,PapZan02,TerVer00} deal with \emph{oscillatory} solutions
of the superlinear indefinite equation \eqref{eq_intro}. On the other hand, a chaotic dynamics
entirely made by positive solutions for an equation like $\ddot u + q(t)g(u) = 0$ is produced
in \cite{BosZan12}, but
in the case of a nonlinearity $g(u)$ of \emph{super-sublinear} type;
the topological technique (the so-called ``Stretching Along the Paths'' method)
used therein, moreover, is not applicable in the present context.

In our result, periodic solutions play a crucial role and, indeed,
bounded non-periodic solutions are constructed as limit of periodic ones
when the period goes to infinity.
In doing this, a careful analysis has to be performed in order to ensure that the constant $\mu^*$
for which $mT$-periodic solutions are available can be chosen independently on the integer $m$,
and thus making possible the passage to the limit $m \to +\infty$. We stress that this procedure
(obtaining chaotic solutions as limit of periodic solutions) does not seem to be easily reproducible
using the approaches proposed in
\cite{BonGomHab05,FelZanPP,GauHabZan03,GauHabZan04,GirGom09,GirGom09b}.

Our proof of the existence of periodic solutions is variational: we exploit the fact that solutions
of \eqref{eqmain_intro} are critical points of the related action functional $J$, and we construct a
suitable natural constraint for such a functional, that is, a constraint for which constrained critical
points of $J$ are free ones. The most famous natural constraint is the \emph{Nehari manifold}, which
can be successfully used in order to find positive solutions of the two-point boundary value problem
for \eqref{eqmain_intro} when restricted to intervals where $a$ is non negative. Indeed,
as shown in \cite{MooNeh59,Neh61}, letting
\[
\mathcal{N}_i=\left\{ u \in H^1_0(I_i^+) : \; u\not\equiv 0, \; \int_{I_i^+} \dot u^2 =
\int_{I_i^+} a^+ u^4 \right\} \qquad\text{ and }\qquad
c_i = \inf_{u \in \mathcal{N}} \frac{1}{4}\int_{I_i^+} \dot u^2,
\]
we have that the set
\begin{equation} \label{eq:intro_nehari}
\mathcal{K}_i = \left\{ u \in \mathcal{N}_i :
\frac{1}{4}\int_{I_i^+} \dot u^2 = c_i
\right\}
\subset W^{2,\infty}(I^+_i)
\end{equation}
consists of one-sign solutions of \eqref{eqmain_intro} with homogeneous Dirichlet boundary
conditions on $I_i^+$. Under this perspective,
one may read Theorem \ref{thmain_intro} as a singular perturbation result, where the singular
limit of the solutions we find, as $\mu\to+\infty$, is the set
\begin{equation} \label{eq:intro_nehari_tante}
\mathcal{K}_{\mathcal{L}} = \left\{u\in W^{1,\infty}(\RR) : u|_{I^+_i} \in\mathcal{K}_i,\,
u|_{I^+_i}>0
\text{ if }\mathcal{L}_i=1\text{ and }u\equiv0\text{ elsewhere}\right\}.
\end{equation}
In fact, we consider a periodic truncation of such a set, and we deform it to a suitable
Nehari-type constraint, showing that minima of the action functional on such a set, when $\mu$
is large, correspond to the desired periodic orbits. The same idea was already exploited in
\cite{BonGomHab05,GirGom09,GirGom09b}, even though such papers concern the PDE setting, with
Dirichlet boundary conditions on bounded domains, and it is not clear how to modify the
arguments there, in order to treat periodic conditions, and to obtain uniform estimates
so that one can pass to the limit to unbounded domains. To overcome this difficulty, we
rely on an abstract result contained in \cite{NorVer13}: this avoids the necessity of
constructing a projection operator to the constraint, which is usually one of the most
delicate parts when dealing with Nehari-type arguments. The pay-off of such a method is
that it provides with sharp localization of the solutions, and it allows to prove optimal
bounds, uniform as $\mu\to+\infty$.
\begin{theorem}\label{thmain_intro_conv}
In the assumptions of Theorem \ref{thmain_intro}, let $\mathcal{L}$ be fixed, and for
every $\mu>\mu^*$ let $u_\mu$ denote the corresponding solution. Then there exists
a constant $C=C(a,\mathcal{L})$, not depending on $\mu$, such that
\[
\|u_\mu\|_{C^{0,1}(\RR)}<C.
\]
Furthermore, up to subsequences,
\[
u_\mu \to \bar u \in \mathcal{K}_{\mathcal{L}},\qquad\text{ in }C^{0,\alpha}(\RR),
\text{ for every }\alpha<1,
\]
and the convergence is true also in $H^1_{\mathrm{loc}}(\RR)$, and in $W^{2,\infty}_{\mathrm{loc}}$
away from the points where the function $a$ changes sign.
\end{theorem}
As we mentioned, the bounds above are optimal: since $u_\mu$ is $C^1$,
were the convergence $C^{0,1}$, also $\bar u$ would be $C^1$; this is impossible, since
the elements of $\mathcal{K}_\mathcal{L}$ cannot be $C^1$ (if $\mathcal{L}_i\neq0$
for at least one $i$).

Variational methods were already successfully exploited to construct entire complex
solutions of nonlinear oscillators in \cite{TerVer00,Ver03,OrtVer04,SoaVer14}. Though
also the methods employed in these papers are related to the results by Nehari, in particular to a broken geodesics
argument, that situation is rather different: the solutions found there are oscillatory, and the
uniform energy estimates to pass from bounded intervals to the real line are obtained through
a control of the distance between consecutive zeroes. Rather, we borrow some ideas from
\cite{TerVer09}, where radial positive multi-bump solutions to a singularly perturbed elliptic
system are investigated.

To conclude, we remark that slight variants of our technique can be invoked
to prove some related results (see Remark \ref{rem:general_results} at the end of the paper).
First of all, one can consider functions $a$ changing sign with some uniform properties, rather
than periodic ones; further, also changing sign solutions can be constructed, by choosing sequences
$\mathcal{L} \in \{-1,0,1\}^{\ZZ}$ and prescribing $u|_{I^+_i}$ to be large and positive (resp.
negative) whenever $\mathcal{L}_i=1$ (resp. $-1$). Moreover, we can prove
the existence of $2^n-1$ positive solutions (or $3^n-1$ nontrivial ones, possibly changing
sign) to the $T$-periodic boundary value problem
(that is, $u(0) = u(T)$ and $\dot u(0) = \dot u(T)$) associated with \eqref{eqmain_intro},
whenever the function $a$ has $n$ intervals of positivity in a period.
This gives a $T$-periodic counterpart of the result first proved by Gaudenzi, Habets and Zanolin
for the Dirichlet problem. The Neumann boundary value problem $\dot u(0) = \dot u(T) = 0$
could be also successfully considered, using very similar arguments, and thus extending \cite{Bos11}.
We stress, however, that all these results, dealing with a boundary value problem on a finite
interval,
can be obtained with much easier arguments (on the lines of the main application in \cite{NorVer13})
with respect to the ones described in this paper, whose crucial theme insists on finding estimates
for the threshold $\mu^*$, which are independent of the size of the considered interval.

\section{Main result and strategy of the proof}\label{sec:mainres}

In this paper we deal with the existence of positive solutions of the superlinear indefinite equation
\begin{equation}\label{eqmain}
\ddot u + a_\mu(t) u^3 = 0,\qquad t\in\RR,
\end{equation}
where $\mu > 0$ is a large parameter and
$$
a_\mu(t + T) = a_\mu(t) := a^+(t) - \mu a^-(t), \qquad \text{for every } t,
$$
with \(a^+(t) = \max(0,a(t))\) and  \(a^-(t) = \max(0,-a(t))\) denoting the positive
and negative part of a sign-changing, $T$-periodic function $a \in L^{\infty}(\RR)$.
For the sake of simplicity, we assume that $a$ changes sign just once in $[0,T]$,
that is:
\begin{itemize}
\item[(A)] there exists $\tau \in \,]0,T[\,$ such that
$$
a(t) \geq 0, \not\equiv 0 \;\text{ on } [0,\tau], \qquad
a(t) \leq 0, \not\equiv 0 \;\text{ on } [\tau,T]
$$
\end{itemize}
(even though we can treat much more general situations, see Remark \ref{rem:general_results}
at the end of the paper). Up to a time-translation and a suitable choice of $\tau$ we can suppose that
\begin{equation}\label{eq:cond_int_estremi}
\int_{\tau}^{\tau+\delta}a^-(t)\,dt >0
\qquad \text{and} \qquad
\int^{T}_{T-\delta}a^-(t)\,dt >0,
\end{equation}
for every small $\delta > 0$.

From now on, we also use the notation
$$
\sigma_i = iT, \quad \tau_i = \tau + iT, \quad I^+_i = [\sigma_i, \tau_i],
\quad I^-_i = [\tau_i,\sigma_{i+1}], \quad \mbox{ for every } i \in \mathbb{Z}.
$$

Our main result reads as follows. In the statement below,
$\Sigmadue$ denotes the space of double-sequences of $0$ and $1$.
Moreover, for $\mathcal{L} = \{\mathcal{L}_i\}_{i \in \mathbb{Z}} \in \Sigmadue$, we set
$$
\textbf{0}_{\mathcal{L}} = \sup \left\{ n \in \mathbb{N} \; : \; \exists i \in \mathbb{Z} \, \mbox{ s.t. }
\mathcal{L}_j = 0, \; \forall j=i,\ldots,i+n-1 \right\},
$$
namely, the maximal length of strings in $\mathcal{L}$ entirely composed by $0$.

\begin{theorem}\label{thmain}
For every integer $k \geq 1$ there exists
$\mu^* =\mu^*(k) > 0$ such that, for every $\mu > \mu^*$
and for every $\mathcal{L} \in \Sigmadue$ with
\begin{equation}\label{ipotesil}
\textbf{0}_{\mathcal{L}} \leq k,
\end{equation}
equation \eqref{eqmain} has a positive solution $u \in W^{2,\infty}(\mathbb{R})$
such that, for every $i \in \ZZ$,
\begin{equation}\label{gobbe}
\int_{I^+_i}\dot u^2 <r^2 \mbox{ if } \; \;\mathcal{L}_i = 0
\qquad \mbox{ and } \qquad
\int_{I^+_i}\dot u^2 >r^2 \mbox{ if } \; \;\mathcal{L}_i = 1,
\end{equation}
and
\begin{equation}\label{eq:C1bounds}
\Vert \dot u \Vert_{L^\infty(\RR)} \leq \rho,
\end{equation}
where $r$ and $\rho$ are positive explicit constants, only depending on the weight $a$.

More precisely: for any $\eps > 0$ there exists $\mu^*=\mu^*(k,\eps) \geq \mu^*(k)$, with
$\mu^*(k,\eps) \to +\infty$ for $\eps \to 0^+$, such that the solution $u$ fulfills,
for any $i \in \mathbb{Z}$,
\begin{itemize}
\item[{(P1)}] $\Vert u \Vert_{L^\infty(I^-_i)} + \int_{I^-_i} \dot u^2 \leq \eps$,
\item[{(P2)}] $\Vert u \Vert_{W^{2,\infty}(I^+_i)} \leq \eps$ if $\mathcal{L}_i = 0$,
\item[{(P3)}] $\mathrm{dist}_{W^{2,\infty}}(u|_{I^+_i},\mathcal{K}_i ) \leq \eps$ if
$\mathcal{L}_i = 1$,
where $\mathcal{K}_i$ is defined in \eqref{eq:intro_nehari}.
\end{itemize}

Finally, the solution can be chosen to be $mT$-periodic whenever the sequence $\mathcal{L}$ is $m$-periodic for some $m \in \mathbb{N}$.
\end{theorem}

The proof of Theorem \ref{thmain} is based on an approximation procedure which we now describe.
For each integer $N \geq 0$, consider the interval
$I_N =[\sigma_{-N},\sigma_{N+1}]$, that is,
\[
I_N = I^+_{-N} \cup I^-_{-N} \cup \ldots \cup I^+_{N} \cup I^-_{N}.
\]
Then, the following result holds true.

\begin{theorem}\label{thper}
For any integer $k \geq 1$, there exists $\mu^* > 0$ such that for every
$\mu > \mu^*$, $\mathcal{L} \in \Sigmadue$
satisfying \eqref{ipotesil} and  $N > k$, equation
\begin{equation}\label{eqper}
\ddot u + a_\mu(t) u^3 = 0, \qquad t \in I_N,
\end{equation}
has a positive solution
$u \in W^{2,\infty}_{\mathrm{per}}(I_N)$
such that properties \eqref{gobbe}, \eqref{eq:C1bounds}, (P1), (P2) and (P3) hold,
for $i = -N, \ldots,N$.
\end{theorem}

Notice that, whenever the sequence $\mathcal{L}$ in
Theorem \ref{thmain} is $m$-periodic, with $m$ an odd integer number,
Theorem \ref{thper} gives the existence of a positive $mT$-periodic solution
to \eqref{eqmain}. The case in which $m$ is an even integer can be handled in a completely analogous way by considering,
for $N \geq 1$, the interval $\widetilde I_N = [\sigma_{-N},\sigma_{N-1}]$.

On the other hand, for non-periodic sequences $\mathcal{L}$, the
corresponding positive solution $u = u_{\mathcal{L}} \in W^{2,\infty}(\mathbb{R})$
of Theorem \ref{thmain}
can be constructed as limit, for $N \to +\infty$, of the solutions
$u_{\mathcal{L},N} \in W^{2,\infty}_{\mathrm{per}}(I_N) $ found in Theorem \ref{thper}.
More details for this (quite standard) argument will be given at the end of the paper,
in Section \ref{secfinale}.
\medbreak
From now on, we will concentrate on Theorem \ref{thper},
whose proof will take a great part of the paper. It relies on a variational argument, consisting in the minimization of the action functional
$$
J_{\mu,I_N}(u) = \frac{1}{2}\int_{I_N}\dot u^2 - \frac{1}{4}\int_{I_N}a_{\mu}u^4
$$
on suitable Nehari-type subset of $H^1_{\mathrm{per}}(I_N)$
(in the following, not to overload the notation, we will often
drop the subscript $I_N$ when no confusion is possible).
To describe our procedure, we need some preliminary notation.

First of all, we define a cut-off function
$\eta \in C^{\infty}_c(\mathbb{R})$ such that $0 \leq \eta(t) \leq 1$, for every $t \in \mathbb{R}$ and
\begin{equation*}
\eta \equiv 1 \; \mbox{ on } [0,\tau] \qquad \mbox{ and } \qquad \eta \equiv 0 \;\mbox{ on } \mathbb{R} \setminus \left[\frac{\tau-T}{4},\tau+\frac{T-\tau}{4}\right].
\end{equation*}
Moreover, we set
$$
\eta_i(t) = \eta(t - \sigma_i), \qquad \mbox{ for every } i \in \mathbb{Z},
$$
so as to obtain a family of cut-off functions $\{\eta_i\}_{i \in \mathbb{Z}}$ such that
$\eta_i \eta_j \equiv 0$ whenever $i \neq j$.

Next, we turn to introduce the Nehari-type constraint, which depends on $N$ and $\mathcal{L}$.
Setting
\[
L = \left\{ i \in \{-N,\ldots,N\} : \mathcal{L}_i=1 \right\},
\]
we define the subspaces
\[
V^+ = \left\{ u \in H^1_{\mathrm{per}}(I_N) : \; u \equiv 0 \mbox{ on } \bigcup_{i \in L}I^+_i \right\},
\]
and, for any $u \in H^1_{\mathrm{per}}(I_N)$,
\[
V^-_u = \mathrm{span}_{i \in L}\{ \eta_i u \}.
\]
The Nehari-type set is then defined as
\begin{equation}\label{nehari}
\mathcal{P}_\mu = \left\{ u \in H^1_{\mathrm{per}}(I_N) \; : \; \mathrm{proj}_{\left(V^+ \oplus V^-_u\right)} \nabla J_{\mu}(u) = 0 \right\}.
\end{equation}
The next result collects some properties enjoyed by the functions in $\mathcal{P}_\mu$.
\begin{lemma}\label{carattN}
Let $u \in \mathcal{P}_\mu$.
Then (all integrals are understood on $I_N$):
\begin{itemize}
\item[(i)] $\int \dot u \dot v = \int a_{\mu} u^3 v$ for every $v \in V^+$. In particular,
$u \in W^{2,\infty}\left(I_N \setminus \cup_{i \in L} I^+_i\right)$
and
$$
\ddot u + a_{\mu}(t)u^3 = 0, \quad \mbox{ on } I_N \setminus \cup_{i \in L} I^+_i.
$$
\item[(ii)] $\int \dot u \dot{(\eta_i u)} = \int a_{\mu}\eta_i u^4$ for every $i\in L$;
equivalently,
$$
\int_{I_i^+} \left(\dot u^2 - a^+ u^4\right) = u(\tau_i)\dot u(\tau_i^+) - u(\sigma_i)\dot u(\sigma_i^-).
$$
Here, $\dot u(\sigma_i^-)$ and $\dot u(\tau_i^+)$ are respectively the left derivative of $u$ in $\sigma_i$ and the right derivative
of $u$ in $\tau_i$, whose existence is guaranteed by the previous point (i).
\item[(iii)] $\int \dot u^2 = \int a_{\mu} u^4$; hence
$$
J_{\mu}(u) = \frac14 \int_{I_N} \dot u^2 = \frac14 \int_{I_N} a_{\mu} u^4.
$$
\item[(iv)] $\int \dot u \dot{(\eta_i^2 u)} = \int a_{\mu} \eta_i^2 u^4$ for every $i= -N,\dots,N$; as a consequence
$$
\int_{I_N} \dot{(\eta_i u)}^2
= \int_{I_N} a_{\mu} \eta_i^2 u^4 + \int_{I_N} \dot{\eta_i}^2 u^2.
$$
\end{itemize}
\end{lemma}

\begin{proof}
We prove separately each point.
\begin{itemize}
\item[\textit{(i)}] This corresponds to $\mathrm{proj}_{V^+}\nabla J_{\mu}(u) = 0$.
In particular, this implies that $u$ solves $\ddot u + a_{\mu}(t)u^3 = 0$
in the sense of distributions on $I_N \setminus \cup_{i \in L} {I}^+_i$,
and the second claim follows by elliptic regularity.
\item[\textit{(ii)}] The first equality corresponds to $\mathrm{proj}_{V^-_u}\nabla J_{\mu}(u) = 0$. Denoting $\mathrm{supp} (\eta_i) = [\sigma_i',\tau_i']$, such an equality writes as
\[
\begin{split}
\int_{I_i^+} \left(\dot u^2 - a^+ u^4\right)  &=
- \int_{[\sigma_i',\sigma_i]\cup[\tau_i,\tau_i']} \dot u \dot{(\eta_i u)} - a_{\mu}\eta_i u^4 \\
& =- \int_{[\sigma_i',\sigma_i]\cup[\tau_i,\tau_i']}
     \left(-\ddot u - a_{\mu} u^3\right) (\eta_i u) +  u(\tau_i)\dot u(\tau_i^+) - u(\sigma_i)\dot u(\sigma_i^-),
\end{split}
\]
and the first term vanishes by {(i)}.
\item[\textit{(iii)}] The identity is equivalent to $\langle \nabla J_{\mu}(u), u \rangle = 0$, then we just need to show that $u \in V^+ \oplus V^-_u$. Let us write $u \in H^1_{\mathrm{per}}(I_N)$ as
$$
u = \sum_{i \in L} \eta_i u +\left(1-\sum_{i \in L} \eta_i  \right)u;
$$
then the first term lies in $V^-_u$, the second one in $V^+$.
\item[\textit{(iv)}] Similarly we need to prove that $\eta^2_iu \in V^+ \oplus V^-_u$, for every $i= -N,\dots,N$.
When $i \notin L$, we immediately conclude that $\eta^2_iu \in V^+$. On the other hand, when $i \in L$ we write
$$
\eta^2_iu = \eta_i(\eta_i - 1)u + \eta_i u;
$$
the first term belongs to $V^+$, while the second one to $V^-_u$.
\end{itemize}
\end{proof}
Notice that, in view of Lemma \ref{carattN} (iii), the functional
$J_\mu$ is bounded below on $\mathcal{P}_\mu$. Our minimization problem will be settled in an open subset
$\mathcal{N}_\mu \subset \mathcal{P}_\mu$.
To describe it, we need to introduce some notation.
First of all, define for some suitable $\zeta \in (0,(T-\tau)/2)$
\begin{equation}\label{localnehari}
\begin{array}{l}
\vspace{0.3cm}
\mathcal{N} = \left\{ u \in H^1_0(0,\tau) : \; u\not\equiv 0, \; \int_0^{\tau} \dot u^2 = \int_0^{\tau} a^+ u^4 \right\} \\
\mathcal{Z}_\zeta =
\left\{ u \in \mathcal{N} \; : \;
\exists \bar{t} \in [\zeta,\tau-\zeta] \mbox{ s.t. } u(\bar{t}) = 0
\right\}.
\end{array}
\end{equation}
and set
\begin{equation}\label{localneharimin}
c = \inf_{u \in \mathcal{N}} \frac{1}{4}\int_0^{\tau} \dot u^2 \qquad \mbox{ and } \qquad c_\zeta = \inf_{u \in \mathcal{Z}_\zeta}
\frac{1}{4}\int_0^{\tau} \dot u^2.
\end{equation}
We claim here that
$$
c < c_\zeta.
$$
Indeed, notice first that $c \leq c_\zeta$ since $\mathcal{Z}_\zeta \subset \mathcal{N}$.
Then, observe that $c$ and $c_\zeta$ are both attained
(as infimum value of the corresponding minimizing problems): this is well known for $c$,
and the same proof works also for $c_\zeta$. Since functions attaining the value $c$
are (non-trivial) constant-sign $H^1_0$-solutions
of $\ddot u + a^+ u^3 = 0$ on $(0,\tau)$ \cite{MooNeh59,Neh61}, which are not in $\mathcal{Z}_\zeta$,
we conclude $c \neq c_\zeta$.

Then, for suitable constants $r,K,\rho > 0$, consider the following conditions
\begin{itemize}
\item[(C1)] $\int_{I^+_i} \dot u^2 < r^2$ for $i \notin L$
and $r^2 < \int_{I^+_i} \dot u^2 < 2 (c + c_\zeta)$ for $i \in L$,
\item[(C2)] $u(t) > 0$ for every $t \in \bigcup_{i \in L}[\sigma_i + \zeta, \tau_i - \zeta]$,
\item[(C3)]\label{(C3)} $\vert u(t) \vert < K$ for every $t \in \bigcup_{i=-N}^N I^-_i$,
\item[(C4)] for $i \in L$,
$$
\begin{array}{llllll}
\dot{u}(\sigma_i^-)<\rho & \mbox{if } u(\sigma_i) \geq 0 & \mbox{ and } & \dot{u}(\sigma_i^-)>-\rho & \mbox{if } u(\sigma_i) \leq 0 \\
\dot{u}(\tau^+_i)>-\rho & \mbox{if } u(\tau_i) \geq 0 & \mbox{ and } & \dot{u}(\tau^+_i)<\rho & \mbox{if } u(\tau_i) \leq 0,
\end{array}
$$
\end{itemize}
and set
\begin{equation}\label{eq:defnehari}
\mathcal{N}_\mu = \left\{ u \in \mathcal{P}_\mu  : \;
u \mbox{ satisfies (C1), (C2), (C3), (C4)} \right\}.
\end{equation}
The precise value for $\zeta$, $r$ and $\rho$ will be given in
\eqref{sceltazeta}, \eqref{eq:r} and \eqref{sceltarho} respectively.
As for $K$, \label{(C1)} the choice is more arbitrary, since (as it will be clear from the proof)
any positive value for it works, up to enlarging $\mu^*$.
However, a natural choice can be made by recalling (see, for instance, \cite[Lemma 4.3]{GauHabZan03}) that
positive solutions to \eqref{eqper} are $L^{\infty}$ a-priori bounded,
independently on both $\mu > 0$ and $N$.

\begin{remark}\label{rem:soluzionilocali}
A few comments on the set $\mathcal{N}_{\mu}$ are now in order. First, the derivatives involved in (C4) are well defined since $u$ solves \eqref{eqper} on each $I^-_i$; furthermore such a condition is open in the $H^1$-topology.
Second, the set $\mathcal{N}_\mu$ is non-empty since, recalling \eqref{eq:intro_nehari_tante},
\[
\mathcal{K}_\mathcal{L} \subset \mathcal{N}_\mu,
\]
provided $r^2 < 4c$.
\end{remark}

As a first step towards Theorem \ref{thper}, we have the following
result, whose proof will be given in Section \ref{secproofnehari},
taking advantage of some technical lemmas developed in
Section \ref{sectecnica}, as well as of the main result in \cite{NorVer13}.

\begin{proposition}\label{neharidef}
There exists $\mu^* > 0$ (depending on the weight function $a$ and on
the integer $k$, but not on $\mathcal{L}$ and $N$) such that, for any $\mu > \mu^*$,
the set $\mathcal{N}_\mu$ is a $C^1$ embedded submanifold of $H^1_{\mathrm{per}}(I_N)$
such that any constrained Palais-Smale sequence is a free one.
That is, if $(u_n) \subset \mathcal{N}_\mu$ is such that $J_\mu(u_n)$ is bounded and
$\nabla_{\mathcal{N}_\mu}J_\mu(u_n) \to 0$, then $\nabla J_\mu(u_n) \to 0$ as well.
\end{proposition}

According to the above result, the argument leading to Theorem \ref{thper}
now proceeds by exhibiting a bounded constrained Palais-Smale sequence $(u_n) \subset \mathcal{N}_\mu$
at level
\begin{equation}\label{consmin}
\inf_{u \in \mathcal{N}_\mu} J_\mu(u).
\end{equation}
Indeed, Proposition \ref{neharidef} implies that this is a free bounded Palais-Smale sequence and,
since the gradient of $J_\mu$ is a compact perturbation of an invertible operator,
a (free) critical point for $J_\mu$ is obtained. This is a solution of \eqref{eqper}
having - by construction - the desired complex behavior. Sections \ref{secvariazione}
and \ref{secexistence} will be devoted to this delicate argument,
which requires a careful understanding of the behavior of $J_\mu$
near the boundary of $\mathcal{N}_\mu$.

To conclude the proof of Theorem \ref{thper}, it will be then enough to prove that the
solution found is positive, uniformly bounded in $C^1$, and it has the required properties
(P1), (P2) and (P3) for $\mu \to +\infty$.
This will be the goal of Section \ref{secfinale}, containing also
some more details for the limit $N \to +\infty$ leading to Theorem
\ref{thmain}.

\section{Some technical results}\label{sectecnica}

Throughout the paper, we will make often use of the following elementary inequality:
\begin{equation}\label{disfondamentale}
\Vert u \Vert_{L^\infty(s_1,s_2)} \leq \min_{[s_1,s_2]}\vert u \vert + \sqrt{s_2 - s_1}
\left(\int_{s_1}^{s_2} \dot u^2 \right)^{1/2}, \quad \mbox{ for every } u \in H^1(s_1,s_2).
\end{equation}
This is a simple consequence of the fundamental theorem of calculus, together with the
Cauchy-Schwartz inequality. Notice also that,
if $u$ vanishes somewhere on $[s_1,s_2]$, then
from \eqref{disfondamentale} we obtain the Sobolev-type inequality
\begin{equation}\label{sobolev}
\Vert u \Vert^2_{L^\infty(s_1,s_2)} \leq (s_2 - s_1) \Vert \dot u \Vert_{L^2(s_1,s_2)}^2
\end{equation}
and the Poincaré-type inequality
\begin{equation}\label{poincare}
\Vert u \Vert^2_{L^2(s_1,s_2)} \leq (s_2 - s_1)^2 \Vert \dot u \Vert_{L^2(s_1,s_2)}^2.
\end{equation}

\subsection{Local estimates on the solutions}

In this section, we collect some useful estimates for solutions
of the differential equation \eqref{eqmain} in an interval of positivity
of the weight function, say
\begin{equation}\label{eq+}
\ddot u + a^+(t)u^3 = 0, \qquad t \in [0,\tau],
\end{equation}
and in an interval of negativity, say
\begin{equation}\label{eq-}
\ddot u - \mu a^-(t)u^3 = 0, \qquad t \in [\tau,T].
\end{equation}
Obviously, in view of the $T$-periodicity of $a$, identical results will
hold true for solutions on the intervals $I^+_i$ and $I^-_i$ for every $i \in \mathbb{Z}$.
Let us also observe once for all that solutions $u$ of \eqref{eq+} are concave (resp., convex)
when $u \geq 0$ (resp., $u \leq 0$), while solutions $u$ of \eqref{eq-} are convex (resp., concave)
when $u \geq0$ (resp., $u \leq 0$) and satisfy
\begin{equation}\label{convess}
\vert u(t) \vert \leq \max(\vert u(\tau) \vert,\vert u(T)\vert),
\quad \mbox{ for every } t \in [\tau,T],
\end{equation}
and
\begin{equation}\label{convess_derivata}
\vert \dot u(t) \vert \leq \max(\vert \dot u(\tau^+) \vert,\vert \dot u(T^-)\vert),
\quad \mbox{ for every } t \in [\tau,T].
\end{equation}
These facts will be used several times without further comments.
\medbreak
The first result is a $C^1$ a-priori estimate in $[0,\tau]$,
given a bound for $\int_0^{\tau}\dot u^2$.
\begin{lemma}\label{ulimitata}
For every $M > 0$ there exists $M' > 0$ such that, for any $u$ solving \eqref{eq+}, it holds
$$
\int_0^{\tau} \dot u^2 \leq M \quad \Longrightarrow \quad \Vert u \Vert_{L^{\infty}(0,\tau)}
+ \Vert \dot u \Vert_{L^{\infty}(0,\tau)} \leq M'.
$$
\end{lemma}

\begin{proof}
To see this, we first observe that it is enough to prove the boundedness of $\Vert u \Vert_{L^{\infty}(0,\tau)}$.
Indeed, since $u$ solves \eqref{eq+}, the boundedness of $\Vert \dot u \Vert_{L^{\infty}(0,\tau)}$ follows from the ones of both $u$ and $\ddot{u}$ in $L^{\infty}$, via the elementary inequality
\begin{equation}\label{diselementare}
\Vert \dot u \Vert_{L^{\infty}(0,\tau)} \leq \frac{2}{\tau}\Vert u \Vert_{L^{\infty}(0,\tau)} + \tau \Vert \ddot u \Vert_{L^{\infty}(0,\tau)}.
\end{equation}
In order to conclude, by virtue of \eqref{disfondamentale}, we just need to show that
$\min_{[0,\tau]}|u|$ is bounded. When $u$ vanishes at some point, we immediately conclude. Otherwise, we assume w.l.o.g. that $u>0$ and we consider the principal eigenvalue
$\lambda_1 = \lambda_1(a^+)$ of the problem
\[
\ddot\varphi + \lambda a^+(t)\varphi = 0, \qquad \varphi \in H^1_0(0,\tau),
\]
together with the corresponding positive eigenfunction, $\varphi_1 = \varphi_1(a^+)$. Testing with $u$ the equation solved by $\varphi_1$ and with $\varphi_1$ the equation solved by $u$, after an integration one obtains
\[
\left[\dot\varphi_1 u-\dot u \varphi_1\right]_{0}^{\tau}
 =
\int_{0}^{\tau} \left(\ddot{\varphi}_1 u - \varphi_1 \ddot u\right) =
\int_{0}^{\tau} a^+ u \varphi_1 \left(u^2 -  \lambda_1\right).
\]
Since the left hand side, $\dot \varphi_1(\tau)u(\tau)-\dot \varphi_1(0)u(0)$, is negative (indeed, $\varphi_1$ is positive and vanishes at the end-points of the interval) we conclude that
\[
\min_{t \in [0,\tau]}|u(t)| \leq \sqrt{\lambda_1},
\]
as desired.
\end{proof}
Our second result deals with the distribution of zeros of a solution on $[0,\tau]$.
To state it precisely, recall the definitions of $c$ and $c_\zeta$
in \eqref{localneharimin} and fix $\zeta > 0$ so small that
\begin{equation}\label{sceltazeta}
2\Vert a^+ \Vert_{L^{\infty}}(c+c_\zeta) \zeta^3 < 1
\end{equation}
(this is clearly possible, since $c_\zeta$ is bounded for $\zeta \to 0^+$).
This will be the value of $\zeta$ used henceforth.
\begin{lemma}\label{pochizeri}
Let $u$ be a nontrivial solution of \eqref{eq+} such that
\begin{equation}\label{livello}
\int_0^\tau \dot u^2 \leq 2(c+c_\zeta).
\end{equation}
Then, $u$ has at most one zero in $[0,\zeta]$ and at most one zero in $[\tau-\zeta,\tau]$.
\end{lemma}
\begin{proof}
Assume, just to fix the ideas, that there are $t_1,t_2 \in [0,\zeta]$,
with $t_1 < t_2$, such that $u(t_1) = u(t_2) = 0$. Then, from \eqref{sobolev}, we have
\[
\Vert u \Vert^4_{L^{\infty}(t_1,t_2)} \leq \zeta^2 \left(\int_{t_1}^{t_2} \dot u^2\right)^2.
\]
Multiplying equation
\eqref{eq+} by $u$ and integrating by parts
on $[t_1,t_2]$, from the above estimate and \eqref{livello}, it follows that
\begin{align*}
\int_{t_1}^{t_2} \dot{u}^2 & = \int_{t_1}^{t_2} a^+ u^4  \leq \Vert a^+ \Vert_{L^{\infty}}\zeta^3
\left(\int_{t_1}^{t_2} \dot{u}^2\right)^2\\
&  \leq \Vert a^+ \Vert_{L^{\infty}} \zeta^3
\int_{0}^{\tau} \dot{u}^2 \int_{t_1}^{t_2} \dot{u}^2 \leq 2\Vert a^+ \Vert_{L^{\infty}}(c+c_\zeta) \zeta^3
\int_{t_1}^{t_2} \dot{u}^2.
\end{align*}
In view of the choice of $\zeta$ in \eqref{sceltazeta}, we obtain a contradiction (since $u\not\equiv 0$).
\end{proof}
In the next results, we study the behavior of solutions on $[\tau,T]$.
Our first lemma in this direction plays a crucial role in the rest of the paper,
describing the behavior of solutions at the boundary of $[\tau,T]$
when the parameter $\mu$ is large.
\begin{lemma}\label{upiccola}
For every $\varepsilon > 0$ and $R > 0$, there exists $\widehat\mu > 0$ such that,
for any $\mu > \widehat\mu$ and for any solution $u$ of \eqref{eq-} such that $\Vert u \Vert_{L^\infty(\tau,T)} \leq K$ (see assumption (C3) at page \pageref{(C3)}),
$$
\left\{
\begin{array}{l}
\dot u(\tau^+) \geq -R \\
u(\tau) \geq 0
\end{array}
\right. \quad \Longrightarrow \quad u(\tau) \leq \varepsilon
\quad \mbox{ and } \quad \left\{
\begin{array}{l}
\dot u(\tau^+) \leq R \\
u(\tau) \leq 0
\end{array}
\right. \quad \Longrightarrow \quad -\varepsilon \leq u(\tau)
$$
(and analogous estimates hold for $u(T)$).
\end{lemma}
\begin{proof}
We will deal only with the case $\dot u (\tau^+) \geq -R$, $u(\tau) \geq 0$, the other being analogous.
Let us fix $\varepsilon >0$ and assume, by contradiction, that $u(\tau) > \eps$. A simple convexity argument implies that (provided $\varepsilon$ is small enough, so that $\varepsilon/R < T- \tau$)
\[
u(t) \geq \eps -R(t-\tau) \quad \text{for every } t \in [\tau,\tau+\varepsilon/R],
\]
and $u(t) \geq \eps/2$ for every $t \in [\tau,\bar t\,]$, with $\bar t = \tau+\varepsilon/(2R)$.
Hence, integrating twice equation \eqref{eq-}, we have
\[
\begin{split}
u(\bar t) &
\geq u(\tau) + \dot u(\tau^+)(\bar t -\tau) +\mu \frac{\eps^3}{8} \int_{\tau}^{\bar t}\int_{\tau}^t a^-(s)ds \, dt \\
& \geq \frac{\eps}{2} +\mu \frac{\eps^3}{8} \int_{\tau}^{\bar t}\int_{\tau}^t a^-(s)ds \, dt,
\end{split}
\]
which contradicts $\Vert u \Vert_{L^\infty(\tau-T,0)} \leq K$ whenever $\mu$ is sufficiently large, since $\int_{\tau}^t a^-(s)ds >0$ by \eqref{eq:cond_int_estremi}.
\end{proof}
Our next result deals with the decay of the solutions, when $\mu \to +\infty$,
on compact subintervals of $(\tau,T)$.
\begin{lemma}\label{decay}
For any $\delta \in (0,(T-\tau)/2)$ there exists $C_\delta > 0$ such that,
for any $u$ solving \eqref{eq-}, it holds that
\begin{equation}\label{paregiusto}
\vert u(t) \vert \leq C_\delta \left( \frac{\max(\vert u(\tau) \vert,\vert u(T)\vert)}{\mu}\right)^{1/3}, \quad \mbox{ for every }
t \in [\tau+\delta,T-\delta].
\end{equation}
\end{lemma}
\begin{proof}
Let us fix $\delta \in (0,(T-\tau)/2)$. By convexity arguments it holds
\[
|u(t)| \leq \max \left(|u(\tau+\delta)|,|u(T-\delta)|\right),
\qquad \text{for every } t \in [\tau+\delta,T-\delta],
\]
hence we need an estimate on $\max \left(|u(\tau+\delta)|,|u(T-\delta)|\right)$.
Let us assume that
\[
\max \left(|u(\tau+\delta)|,|u(T-\delta)|\right) = |u(T-\delta)|,
\]
the opposite case being the same. Under this hypothesis it holds
\[
u(T-\delta) \geq 0 \;(\text{resp.}\leq 0)
\quad \implies \quad
\dot u(T-\delta) \geq 0 \;(\text{resp.}\leq 0).
\]
We will detail the proof in the case $u(T-\delta) \geq 0$; first of all we remark that $u(t) \geq u(T-\delta)$ on $[T-\delta,T]$.
Integrating twice equation \eqref{eq-} we have
\[
\begin{split}
u(T)
& \geq u(T-\delta) + \dot u(T-\delta)\delta +\mu u^3(T-\delta) \int_{T-\delta}^{T} \int_{T-\delta}^{t} a^-(s)ds\,dt \\
& \geq \mu u^3(T-\delta) \int_{T-\delta}^{T} \int_{T-\delta}^{t} a^-(s)ds\,dt.
\end{split}
\]
Being the above integral strictly positive by \eqref{eq:cond_int_estremi}, we take
\(
C_\delta = \left(\int_{T-\delta}^{T} \int_{T-\delta}^{t} a^-(s)ds\,dt\right)^{-1/3}.
\)
\end{proof}

\subsection{Some local Nehari-type arguments}

In this section, we collect some results for $H^1$ functions which are ``almost'' in
$\mathcal{N}$ (recall the definition \eqref{localnehari}) in the sense
that $\left\vert\int_0^\tau \left(\dot u^2 - a^+ u^4\right)\right\vert$ and
$\vert u(0) \vert, \vert u(\tau) \vert$ are small. We fix here
\begin{equation}\label{eq:r}
r = \left( 32 \Vert a^+ \Vert_{L^\infty} \tau^3\right)^{-1/2}
\end{equation}
and this will be the value of $r$ used throughout the paper.
\medbreak
Our first lemma says that, for functions almost
in $\mathcal{N}$, the value $\int_0^\tau \dot u^2$
is either very small or quite large.
It is worth noticing that from its proof we can conclude
that $r^2 < 4c$ (compare with Remark \ref{rem:soluzionilocali}).
\begin{lemma}\label{nehari1}
For every $\varepsilon > 0$ small, there exists a constant
$\delta_\eps > 0$ and with $\delta_\eps \to 0$ for $\eps \to 0^+$
such that,
for any $u \in H^1(0,\tau)$,
$$
\left\{
\begin{array}{l}
\vspace{0.2cm}
\displaystyle{\left\vert \int_0^\tau \left(\dot u^2 - a^+ u^4\right) \right\vert\leq \varepsilon} \\
\displaystyle{\vert u(0) \vert, \vert u(\tau) \vert \leq \varepsilon}
\end{array}
\right.
\quad
\Longrightarrow
\quad
\left\{
\begin{array}{ll}
\vspace{0.2cm}
\displaystyle{\int_0^\tau \dot u^2 \leq \delta_\eps} & \; \mbox{if } \quad \displaystyle{\int_0^\tau \dot u^2 \leq r^2} \\
\displaystyle{\int_0^\tau \dot u^2 \geq 2r^2} & \; \mbox{if } \quad \displaystyle{\int_0^\tau \dot u^2 \geq r^2.}
\end{array}
\right.
$$
\end{lemma}

\begin{proof}
Assume that, for some $\varepsilon > 0$, the conditions on the left-hand side hold true.
Then from \eqref{eq:cond_int_estremi}, it holds
$$
\Vert u \Vert_{L^\infty(0,\tau)} \leq \varepsilon
+ \sqrt{\tau}\left( \int_{0}^{\tau} \dot u^2 \right)^{1/2};
$$
hence, recalling the elementary inequality $(A+B)^4 \leq 8(A^4+B^4)$, $A,B \geq 0$,
\[
\int_{0}^{\tau} a^+u^4 \leq
8 \|a^+\|_{L^\infty} \left[\tau \eps^4 + \tau^3 \left( \int_{0}^{\tau} \dot u^2 \right)^{2}\right].
\]
Summing up,
\[
\int_{0}^{\tau} \dot{u}^2 \leq C_\eps
+ C \left( \int_0^\tau \dot u^2 \right)^{2},
\]
where $C_\eps = \eps +
8 \|a^+\|_{L^\infty} \tau \eps^4$ and $C =8 \|a^+\|_{L^\infty} {\tau}^3$. Solving the second order inequality in $\int_{0}^{\tau} \dot u^2$ we obtain
\[
\text{either }\qquad\int_0^{\tau} \dot u^2 \leq \delta^-_\eps
\qquad \text{or} \qquad
\int_0^{\tau} \dot u^2 \geq \delta^+_\eps,
\]
where
\[
\delta^{\pm}_\eps = (1\pm \sqrt{1-4C C_\eps})/(2C).
\]
Since, for $\eps \to 0^+$, $\delta^-_\eps \to 0$ and $\delta^+_\eps \geq 1/(2C)=2r^2$, the thesis follows.
\end{proof}

Our second result concerns the presence of internal zeros for
functions almost in $\mathcal{N}$.

\begin{lemma}\label{nehari2}
Let $\zeta$ be as in \eqref{sceltazeta}. Then there exist $\bar\eps > 0$ and $\omega > 0$ such that,
for any $u \in H^1(0,\tau)$,
$$
\left\{
\begin{array}{l}
\vspace{0.2cm}
\displaystyle{\left\vert \int_0^\tau \left(\dot u^2 - a^+ u^4\right) \right\vert \leq \bar\eps} \\
\vspace{0.2cm}
\displaystyle{\vert u(0) \vert, \vert u(\tau) \vert \leq \bar\eps} \\
\displaystyle{r^2 \leq \int_0^\tau \dot u^2 \leq 2(c + c_\zeta)}
\end{array}
\right.
\quad
\Longrightarrow
\quad
\vert u(t) \vert \geq \omega \quad \mbox{ for every } \; t \in [\zeta, \tau - \zeta].
$$
\end{lemma}

\begin{proof}
Assume by contradiction that there
exists a sequence $(u_n) \subset H^1(0,\tau)$ with
\begin{equation}\label{nehari2ass}
\left\vert \int_0^\tau \left(\dot u_n^2 - a^+ u_n^4\right) \right\vert \to 0, \quad
\vert u_n(0) \vert, \vert u_n(\tau) \vert \to 0, \quad
r^2 \leq \int_0^\tau \dot u_n^2 \leq 2(c + c_\zeta)
\end{equation}
and such that
$$
u_n(t_n) \to 0, \quad \mbox{ for some } t_n \in [\zeta,\tau-\zeta].
$$
From the second and the third condition in \eqref{nehari2ass}, we infer (once more using \eqref{disfondamentale})
that $u_n$ is bounded in $H^1(0,\tau)$, so that there exists
$\bar{u} \in H^1(0,\tau)$ such that $u_n \to \bar{u}$ weakly in $H^1$ and uniformly.
Moreover, $\bar{u} \in H^1_0(0,\tau)$, $\bar{u}(t^*) = 0$
for a suitable $t^* \in [\zeta,\tau-\zeta]$,
and $\bar{u} \not\equiv 0$ (since from the first and the third condition in \eqref{nehari2ass}
we know that $\int_0^\tau a^+ u_n^4 \geq r^2/2$, for $n$ large). Finally,
$$
\int_0^\tau \dot{\bar{u}}^2 \leq \liminf_{n \to +\infty}\int_0^\tau \dot{u}_n^2 \leq \liminf_{n \to +\infty}
\int_0^\tau a^+ u_n^4 = \int_0^{\tau} a^+ \bar{u}^4.
$$
Set now
$$
\bar{\lambda} = \left( \frac{\int_0^\tau \dot{\bar{u}}^2}{\int_0^\tau a^+ \bar{u}^4}\right)^{1/2} \leq 1;
$$
then the function $\hat{u} = \bar{\lambda} \bar{u}$ lies in $\mathcal{Z}_{\zeta}$.
Hence
$$
4 c_\zeta \leq \int_0^\tau \dot{\hat{u}}^2 \leq \int_0^\tau \dot{\bar{u}}^2 \leq
\liminf_{n \to +\infty}\int_0^\tau \dot u_n^2 \leq 2 (c + c_\zeta) < 4 c_\zeta,
$$
a contradiction.
\end{proof}

\section{Some properties of $\mathcal{N}_\mu$ and proof of Proposition \ref{neharidef}}\label{secproofnehari}

In this section we establish some further fundamental properties
enjoyed by functions in $\mathcal{N}_\mu$, as defined in equation \eqref{eq:defnehari},
and, as a consequence, we give the proof of
Proposition \ref{neharidef}.

Contrarily to the properties in Lemma \ref{carattN},
the fact that the parameter $\mu$ is large
now plays a role. We point out that the final value $\mu^*$ in Theorem
\ref{thper} will be the result of many successive enlargements;
with some abuse of notation, but not to overload it,
in all the subsequent results we will always use the same symbol $\mu^*$
to denote the outcome at each step.
The crucial point, however, is that
the rule for any of these enlargements
depends ultimately only on
the local estimate given in Lemma \ref{upiccola}
(that is, on the value $\widehat\mu$ given therein). For this reason,
the final value $\mu^*$ is a quantity independent on both the sequence $\mathcal{L}$
and on the integer $N$.

\begin{lemma}\label{nehariprimeprop}
For every $\eps > 0$ there exists $\mu^* > 0$ such that, for any $\mu > \mu^*$ and
$u \in \mathcal{N}_\mu$, the following hold true, for any $i \in \{-N,\ldots,N\}$:
\begin{itemize}
\item[(i)] $\vert u(\sigma_i) \vert + \vert u(\tau_i) \vert \leq \eps$,
\item[(ii)] $| \dot u (\sigma_i^-) u (\sigma_i)| + |\dot  u (\tau_i^+) u (\tau_i)| \leq \varepsilon$,
\item[(iii)] $\left\vert \int_{I^+_i} \left(\dot u^2 - a^+ u^4\right) \right\vert \leq \eps$.
\end{itemize}
\end{lemma}

\begin{proof}
As for (i), when $i \in L$ assumption (C4) holds and the conclusion follows directly from Lemma \ref{upiccola},
with the choice $R = \rho$. On the other hand, when $i \notin L$
we first apply Lemma \ref{ulimitata} with $M = r^2$ to ensure that
\begin{equation}\label{bounddotu}
\Vert \dot u \Vert_{L^{\infty}({I^+_i})} \leq M'
\end{equation}
for a suitable $M' > 0$. Second, we notice that $\dot u(\sigma_i^-) = \dot u(\sigma_i^+)$
and $\dot u(\tau_i^+) = \dot u(\tau_i^-)$ since,
by Lemma \ref{carattN} (i), $u$ solves the differential equation
both at $\sigma_i$ and $\tau_i$.
Hence, the conclusion follows again from Lemma \ref{upiccola}, with
the choice $R = M'$.

We now deal with (ii); again, we have to distinguish the cases $i \notin L$ and $i \in L$.
In the former, the conclusion is immediate, recalling
\eqref{bounddotu} (together with $\dot u(\sigma_i^-) = \dot u(\sigma_i^+)$
and $\dot u(\tau_i^+) = \dot u(\tau_i^-)$) and the previous point (i).
In the latter, consider for instance the case $u(\tau_i) > 0$ (the other being analogous). Since $u(\tau_i)$ is small by the previous step, and $\dot u(\tau_i^+) > -\rho$, we just have to prove that $\dot u(\tau_i^+)$ is bounded from above. Since $u$ is convex on ${I^-_i}$ and $u$ satisfies (C3), we have
\[
\dot u (\tau_i^+) \leq \frac{u(\sigma_{i+1})-u(\tau_i)}{\sigma_{i+1}-\tau_i}
\leq \frac{K}{T-\tau}
\]
and the conclusion follows.

Finally, we notice that for every index $i$ it holds
\[
\int_{I^+_i} \left(\dot u^2 - a^+ u^4\right) = u(\tau_i)\dot u(\tau_i^+) - u(\sigma_i)\dot u(\sigma_i^-).
\]
This equality comes from Lemma \ref{carattN} (iii)
if $i \in L$, while it is a consequence of the integration by parts rule when
$i \notin L$. Then, (iii) follows from (ii).
\end{proof}

\begin{lemma}\label{nehariprimeprop2}
For every $\eps > 0$ small enough there exists $\mu^* > 0$ such that, for any $\mu > \mu^*$ and
$u \in \mathcal{N}_\mu$, the following hold true:
\begin{itemize}
\item[(i)] $\Vert u \Vert_{L^\infty({I^-_i})} + \int_{I^-_i} \dot u^2
\leq \eps$, for any $i \in \{-N,\ldots,N\}$,
\item[(ii)] $\Vert u \Vert_{L^\infty({I^+_i})} + \Vert \dot u \Vert_{L^\infty({I^+_i})} +
\Vert \ddot u \Vert_{L^\infty({I^+_i})}  \leq \eps$, for any $i \notin L$,
\item[(iii)] $\int_{I^+_i} \dot u^2 \geq 2r^2$, for any $i \in L$,
\item[(iv)] for a suitable $\omega > 0$,
$$
u(t) \geq \omega, \quad \mbox{ for every } t \in \bigcup_{i \in L}[\sigma_i + \zeta, \tau_i - \zeta].
$$
\end{itemize}
\end{lemma}

\begin{proof}
Integrating by parts equation \eqref{eq-} solved by $u$ on $I^-_i$, we have
\[
\begin{split}
\int_{I^-_i} \dot u^2 & = -\mu \int_{I^-_i} a^- u^4 + u(\sigma_{i+1})\dot u(\sigma_{i+1}^-) -
u(\tau_{i})\dot u(\tau_{i}^+) \\
& \leq u(\sigma_{i+1})\dot u(\sigma_{i+1}^-) -
u(\tau_{i})\dot u(\tau_{i}^+).
\end{split}
\]
Recalling \eqref{convess}, the conclusion (i) follows from the points (i) and (ii) of Lemma \ref{nehariprimeprop}.

We now deal with (ii).
The elementary inequality \eqref{disfondamentale}
together with Lemma \ref{nehari1} and the point (i) of Lemma \ref{nehariprimeprop}
imply that $\Vert u \Vert_{L^{\infty}({I^+_i})}$ can be made arbitrarily small for $\mu$ large.
Since $u$ solves an equation independent of $\mu$ on such an interval, we deduce the bound for
$\Vert \ddot u \Vert_{L^{\infty}({I^+_i})}$.
Using \eqref{diselementare}, we conclude that the same is true for
$\Vert \dot u \Vert_{L^{\infty}({I^+_i})}$, as well.

Finally, (iii) and (iv) follow from Lemmas \ref{nehari1} and \ref{nehari2}, respectively.
\end{proof}

We are now in position to prove Proposition \ref{neharidef}. As we mentioned, such a proof relies
on the main result of \cite{NorVer13}, which we report here for the reader's convenience.
\begin{theorem}[{\cite[Theorems 1.2, 2.8]{NorVer13}}]\label{teo:nv}
Let $X$ be a Hilbert space, $J\in C^2(X,\RR)$, $V^+\subset X$ a fixed closed
linear subspace. We define
\[
V^+_x\equiv V^+,\qquad V^-_x=\mathrm{span}\left\{\xi_1(x),\dots,\xi_h(x)\right\},\qquad V_x = V^+_x \oplus V^-_x,
\]
with $\xi_i \in C^1(\mathcal{A},X)$ for every $i=1,\dots,h$, $\mathcal{A}\subset X$ open, in such a way that
$V_x$ is a proper subspace
for every $x$. Let
\[
\mathcal{M} =\left\{x\in \mathcal{A}:\,\mathrm{proj}_{V_x}\nabla J(x)=0 \right\}.
\]
Let us suppose that for some $0<\delta'<\delta''$ it holds, for every $x\in\mathcal{M}$,
\begin{enumerate}
 \item[(i)] $\|\xi_i(x)\|_X\geq\delta'$, $\langle \xi_i(x),\xi_j(x)\rangle_X=0$,
 for every $i\neq j$;
 \item[(ii)] $\xi'_i(x)[v]\in V_x$ for every $i$ and $v\in V_x$;
 \item[(iii)] $\pm J''(x)[v,v] \geq \delta' \|v\|_X^2$ for every $v\in V^\pm_x$;
 \item[(iv)] $\|\xi'_i(x)[u]\|_X\leq \delta''\|u\|_X$, $|J'(x)[u]|\leq \delta''\|u\|_X$ and
 $|J''(x)[u,w]| \leq \delta'' \|u\|_X\|w\|_X$ for every  $u,w\in X$.
\end{enumerate}
Then the set $\mathcal{M}$ is a $C^1$ embedded submanifold of $X$ such that any constrained
Palais-Smale sequence for $J$ is a free one.
\end{theorem}

\begin{proof}[Proof of Proposition \ref{neharidef}]
We argue in two steps. First of all,
we consider the open subset of $H^1_{\mathrm{per}}(I_N)$ defined as
$$
\mathcal{A} = \left\{ u \in H^1_{\mathrm{per}}(I_N) \; : \;
\begin{array}{l}
\int_{I^+_i} u^2 < \delta^2, \; \mbox{ for } \, i \notin L
\\
\int_{I^+_i} \dot u^2 > r^2, \; \mbox{ for } \, i \in L
\\
3\int_{I_N} \dot\eta_i^2 u^2 < r^2, \; \mbox{ for } \, i \in L
\end{array} \right\},
$$
where
$$
\delta = \min\left\{ r, \frac12 \left[ 3\,T\, (k+1) \|a^+\|_{L^\infty} \right]^{-1/2} \right\},
$$
and, recalling definition \eqref{nehari}, we set
$$
\mathcal{M}_\mu = \mathcal{A} \cap \mathcal{P}_\mu.
$$
Then, we have the following.
\smallbreak
\noindent
\emph{\underline{Step 1}. For any $\mu > 0$, the set $\mathcal{M}_\mu$ is a $C^1$ embedded submanifold of $H^1_{\mathrm{per}}(I_N)$
such that any constrained Palais-Smale sequence is a free one.}
\smallbreak
\noindent
Of course, this comes from a direct application of Theorem \ref{teo:nv}, with $X=H^1_{\mathrm{per}}(I_N)$,
$J=J_\mu$, $\mathcal{M} = \mathcal{M}_\mu$, and $\xi_i(u) = \eta_i u$, $i\in L$ (recall \eqref{nehari}).
From this point of view, assumptions (i), (ii) and (iv) are almost straightforward, and we refer to the proof of
\cite[Theorem 3.5]{NorVer13} for full details. Here we focus on the proof of (iii).

Let $u \in \mathcal{M}_\mu$ and $v \in V^+$; in particular, $v$ vanishes on $\cup _{i \in L}I^+_i$. We have
\[
J_{\mu}''(u)[v,v] = \int_{I_N} \dot v^2 - 3\int_{I_N} a_{\mu}u^ 2v^2 \geq \int_{I_N} \dot v^2 - 3\int_{I_N} a^+u^ 2v^2.
\]
We now indicate with $\Delta_l$ the connected components of $I_N \setminus \cup _{i \in L}I^+_i$ (the index $l$ varies between 1 and the cardinality of $L$, $|L|$, or between 1 and $|L|+1$).
Recalling assumption \eqref{ipotesil}, the following estimate on the length of each $\Delta_l$ holds
\[
|\Delta_l| \leq k T + (T-\tau) \leq (k+1)T.
\]
Hence, since $v$ vanishes at the end-points of each $\Delta_l$, once more using \eqref{disfondamentale}, for every $k$ it holds
\[
\| v\|^2_{L^{\infty}(\Delta_l)} \leq |\Delta_l| \int_{\Delta_l} \dot v^2
\leq (k + 1)T \int_{\Delta_l} \dot v^2,
\]
and hence
\[
\begin{split}
\int_{I_N} a^+u^2v^2 &\leq \sum_l \int_{\Delta_l} a^+u^2v^2\leq \sum_l (k + 1)T \int_{\Delta_l} \dot v^2 \int_{\Delta_l} a^+ u^2 \\
 & \leq (k + 1)T \cdot \|a^+\|_{L^\infty} \cdot k\delta^2 \sum_l \int_{\Delta_l} \dot v^2\\
 & = k(k + 1)T  \|a^+\|_{L^\infty}\, \delta^2 \int_{I_N} \dot v^2,
\end{split}
\]
indeed at most $k$ intervals $I^+_i$ with $i \notin L$ stay in each $\Delta_l$. Continuing the
previous estimates on the second derivative of $J_\mu$ we obtain
\[
\begin{split}
J_{\mu}''(u)[v,v] &\geq
 \left( 1-3k(k+1)T\|a^+\|_{L^\infty}\,\delta^2 \right)
\int_{I_N} \dot v^2\\
& \geq
\left(1+(k+1)^2T^2\right)^{-1}\left( 1-3\,(k+1)T\|a^+\|_{L^\infty}\,k\,\delta^2 \right) \|v\|^2_{H^1(I_N)}\\
&=\delta'\|v\|^2_{H^1(I_N)}
\end{split}
\]
by the Poincaré inequality \eqref{poincare}, where $\delta'>0$ by the definition of $\delta$.

On the other hand, let $v \in V^-_u$.
Then $v = \sum_{i \in L}\alpha_i \eta_i u$ for some $\alpha_i \in \mathbb{R}$ and, by the assumptions on $\eta_i$,
$\dot v^2 = \sum_{i \in L}\alpha_i^2  \dot{(\eta_i u)}^2$.
Using Lemma \ref{carattN} {(iv)}, we obtain
\begin{align*}
J_{\mu}''(u)[v,v] = & \int_{I_N} \dot v^2 - 3\int_{I_N} a_{\mu}(t)u^ 2v^2 = \sum_{i \in L} \alpha_i^2 \left( \int_{I_N}\dot{(\eta_i u)}^2 - 3\int_{I_N} a_\mu\eta_i^2 u^4  \right) \\
= & \sum_{i \in L} \alpha_i^2 \left( -2\int_{I_N} \dot{(\eta_i u)}^2 + 3 \int_{I_N} \dot\eta_i^2 u^2 \right) \\
= & -\int_{I_N} \dot v^2 + \sum_{i \in L}\alpha_i^2\left( 3\int_{I_N}\dot{\eta_i}^2 u^2 - \int_{I_N}\dot{(\eta_i u)}^2\right).
\end{align*}
Since, for every $i \in L$, $u\in\mathcal{M}_\mu$,
\[
3\int_{I_N}\dot{\eta_i}^2 u^2
< r^2
< \int_{I_i^+}\dot u^2
= \int_{I_i^+}\dot{(\eta_i u)}^2
< \int_{I_N}\dot{(\eta_i u)}^2,
\]
the conclusion follows again by Poincar\'e inequality.
\smallbreak
\noindent
\emph{\underline{Step 2}. There exists $\mu^* > 0$ such that, for $\mu > \mu^*$, $\mathcal{N}_\mu \subset \mathcal{M}_\mu$ open.}
\smallbreak
\noindent
According to the definition of $\mathcal{A}$, we have to verify that, for $\mu$ large,
$$
\int_{I^+_i}  u^2 < \delta^2, \quad \mbox{ for } \; i \notin L,
\quad \mbox{ and } \quad
3\int_{I_N} \dot\eta_i^2 u^2 < r^2, \quad \mbox{ for } \; i \in L;
$$
the first inequality follows directly from the smallness of $\Vert u \Vert_{L^\infty(I^+_i)}$
proved in Lemma \ref{nehariprimeprop2} {(ii)}. In the second one the integral is indeed on the
support of $\dot \eta_i$ (which is contained in $I^-_{i-1}\cup I^-_i$), hence the thesis follows from
the smallness of $\Vert u \Vert_{L^\infty(I^-_i)}$ proved in Lemma \ref{nehariprimeprop2} {(i)}.
Notice that the value $\mu^*$ obtained in this way depends (on the function $a$ and) on $k$,
via $\delta$.
\end{proof}

\section{Construction of a local variation}\label{secvariazione}

In this section, we collect some results which will be used in the next Section
\ref{secexistence} to construct local variations
(in $\mathcal{N}_\mu$) for functions
in $\partial\mathcal{N}_\mu$.

\subsection{Variations on $[0,\tau]$}\label{subs:+}

Let $r,\zeta$ be fixed as in \eqref{eq:r} and \eqref{sceltazeta}. Take $p,q: [0,\tau] \to \mathbb{R}$ as the affine functions such that
\begin{equation}\label{defpq}
q(\tau) = p(0) = 0, \qquad q(0) = p(\tau) = 1,
\end{equation}
and define $\rho > 0$ as
\begin{equation}\label{sceltarho}
\rho = 16\sup\left\{ \frac{2K}{T-\tau} + \left\vert\int_0^\tau \left( \dot u \dot p - a^+ u^3 p\right) \right\vert +
\left\vert\int_0^\tau \left( \dot u \dot q - a^+ u^3 q\right)\right\vert \; : \;
\begin{array}{l}
u \in H^1(0,\tau)\smallskip\\
\vert u(0) \vert, \vert u(\tau) \vert \leq 1\smallskip\\
\int_0^\tau \dot u^2 \leq 2(c+c_\zeta)
\end{array}
\right\}.
\end{equation}
\begin{remark}\label{rem:K_soddisfa_rho}
Choosing as a test function in the supremum above any element of $\mathcal{K}_0$ (recall
definitions \eqref{eq:intro_nehari}, \eqref{eq:intro_nehari_tante}),
an integration by parts yields $|\dot u(0^+)|+|\dot u(\tau^-)|<\rho$ for every $u\in
\mathcal{K}_0$, and thus
\[
|\dot u(t^\pm)|\leq\rho\qquad\text{ for every }u\in\mathcal{K}_\mathcal{L},\;t\in\RR.
\]
\end{remark}
The next result is quite technical but it
greatly simplifies the exposition of
some arguments in Section \ref{secexistence}.

\begin{lemma}\label{localvar}
There exists $\eps > 0$ such that, for any
$\alpha_0', \beta_0' \in [-\rho,\rho]$ and
$u \in H^1(0,\tau)$ such that
\begin{equation}\label{ipotesilem1}
\left\{
\begin{array}{l}
\vspace{0.2cm}
\displaystyle{r^2 \leq \int_0^\tau \dot u^2 \leq 2(c + c_\zeta)} \\
\vspace{0.2cm}
\displaystyle{\vert u(0) \vert, \vert u(\tau) \vert \leq \varepsilon} \\
\displaystyle{\int_0^{\tau} \left( \dot u^2 - a^+ u^4\right)
= u(\tau) \beta_0' - u(0) \alpha_0'},
\end{array}
\right.
\end{equation}
there is a $C^1$-map defined on a neighborhood $O$ of $P_0=
(u(0),u(\tau),\alpha_0',\beta_0')$
$$
U = U(\alpha,\beta,\alpha',\beta'): O \to H^1(0,\tau),
$$
such that:
\begin{itemize}
\item[(i)] $U(P_0)(t) = u(t)$ for any $t \in [0,\tau]$ and, for every $(\alpha,\beta,\alpha',\beta') \in O$,
$$
U(\alpha,\beta,\alpha',\beta')(0) = \alpha, \qquad U(\alpha,\beta,\alpha',\beta')(\tau) = \beta,
$$
\item[(ii)] for every $(\alpha,\beta,\alpha',\beta') \in O$,
$$
\int_0^{\tau} \left( \dot U^2(\alpha,\beta,\alpha',\beta') - a^+ U^4(\alpha,\beta,\alpha',\beta')\right) = \beta \beta' - \alpha \alpha',
$$
\item[(iii)] for a suitable $C > 0$, depending only on $r,\zeta,\rho,\eps$, it holds
\begin{equation}\label{formula0}
\left\Vert \nabla_{(\alpha,\beta,\alpha',\beta')} U(P_0) \right\Vert_{H^1(0,\tau)} \leq C.
\end{equation}
\end{itemize}
Moreover, there exists $C_\eps > 0$, depending only on $r,\zeta,\rho,\eps$ and
such that $C_\eps \to 0$ for $\eps \to 0^+$,
such that
\begin{equation}\label{formula1}
\left\vert \partial_\alpha J_{[0,\tau]}(U(P_0))\right\vert \leq
\frac{\rho}{8},
\qquad
\left\vert \partial_\beta J_{[0,\tau]}(U(P_0))\right\vert \leq
\frac{\rho}{8},
\end{equation}
and
\begin{equation}\label{formula3}
\left\vert \partial_{\alpha'} J_{[0,\tau]}(U(P_0))\right\vert \leq C_\eps \vert u(0) \vert,
\qquad
\left\vert \partial_{\beta'} J_{[0,\tau]}(U(P_0))\right\vert \leq C_\eps \vert u(\tau) \vert
\end{equation}
(where $J_{[0,\tau]}(u)=\int_0^\tau\frac12 \dot u^2-\frac14 a^+ u^4$).
\end{lemma}

\begin{proof}
We define
\begin{equation}\label{implicita}
F(\alpha,\beta,\alpha',\beta',\lambda) = \int_0^\tau\left( \dot V^2(t;\alpha,\beta,\lambda) - a^+ V^4(t;\alpha,\beta,\lambda)\right) -
(\beta\beta' - \alpha\alpha'),
\end{equation}
where, for $t \in [0,\tau]$,
$$
V(t;\alpha,\beta,\lambda) = \alpha q(t) + \beta p(t) + \lambda  \left[u(t) - u(0)q(t) - u(\tau)p(t)\right].
$$
By construction,
$$
V(t;u(0),u(\tau),1) = u(t), \quad V(0;\alpha,\beta,\lambda) = \alpha, \quad V(\tau;\alpha,\beta,\lambda) = \beta;
$$
moreover,  $F(P_0,1) = 0$.
Simple calculations show that
$$
\partial_\lambda F(P_0,1) = \int_0^\tau\left(2 \dot u^2 - 4a^+ u^4\right)-\chi_1(u),
$$
where
$$
\chi_1(u) = u(0) \int_0^\tau\left( 2 \dot u \dot q - 4 a^+ u^3 q\right) +
u(\tau) \int_0^\tau\left( 2 \dot u \dot p - 4 a^+ u^3 p\right).
$$
Since
$$
\int_0^\tau\left( 2 \dot u^2 - 4 a^+ u^4\right) = -2 \int_0^\tau \dot u^2 + 4 \left( u(\tau)\beta_0' - u(0)\alpha_0'\right),
$$
and, for $\eps \to 0^+$,
$$
\vert u(\tau)\beta_0' - u(0)\alpha_0' \vert + \vert \chi_1(u) \vert \to 0
$$
uniformly in the class of functions satisfying \eqref{ipotesilem1},
it holds that
$$
\partial_\lambda F(P_0,1) \leq -2r^2 +o_\eps(1)\leq -r^2,
$$
whenever $\eps > 0$ is small enough (depending on $r,\zeta,\rho$, but not on $u$).
Hence, the implicit function theorem applies,
yielding the existence of
$\lambda = \lambda(\alpha,\beta,\alpha',\beta')$
solving implicitly $F(\alpha,\beta,\alpha',\beta',\lambda) = 0$ near $(P_0,1)$.
Then, with the position
$$
U(\alpha,\beta,\alpha',\beta')(t) = V(t;\alpha,\beta,\lambda(\alpha,\beta,\alpha',\beta'))
$$
the points (i) and (ii) of the statement are proved.
We also notice that
$$
\partial_\alpha \lambda(P_0) = -\frac{\alpha'_0 + \int_0^\tau\left( 2\dot u \dot q - 4 a^+ u^3 q\right)}
{\partial_\lambda F(P_0,1)} \quad \mbox{ and } \quad
\partial_{\alpha'} \lambda (P_0)= -\frac{u(0)}{\partial_\lambda F(P_0,1)};
$$
hence
\begin{equation}\label{derivatelambda}
\left\vert \partial_\alpha \lambda(P_0) \right\vert \leq C' \quad \mbox{ and } \quad
\left\vert \partial_{\alpha'} \lambda(P_0) \right\vert \leq C'\vert u(0) \vert,
\end{equation}
where $C' > 0$ is a suitable constant depending only on $r,\zeta,\rho,\eps$.
Of course, similar estimates hold for the derivatives with respect to $\beta$ and $\beta'$.
From this consideration, \eqref{formula0} immediately follows.

To prove the final part of the statement, we observe that
\begin{align*}
\frac{\partial}{\partial\alpha}J_{[0,\tau]}(U)\vert_{P_0} & =
 \int_0^\tau \left[\dot U \partial_\alpha \dot U - a^+U^3 \partial_\alpha U \right]_{P_0}\\
& =
 \int_0^\tau \left[ \dot V \left(\partial_\alpha \dot V + \partial_\alpha \lambda \partial_\lambda \dot V \right)
 - a^+V^3\left(\partial_\alpha  V + \partial_\alpha \lambda \partial_\lambda  V \right) \right]_{P_0}
   \\
& = \partial_\alpha \lambda(P_0)  \int_0^\tau \left( \dot V \partial_\lambda \dot V - a^+ V^3 \partial_\lambda V\right)_{P_0}
+ \int_0^\tau \left( \dot V \partial_\alpha \dot V - a^+ V^3 \partial_\alpha V \right)_{P_0} \\
& = \partial_\alpha \lambda(P_0) \left[ \int_0^\tau \left( \dot u^2 - a^+ u^4\right) - \chi_2(u)\right] +
\int_0^\tau \left( \dot u \dot q - a^+ u^3 q\right),
\end{align*}
where
$$
\chi_2(u) = u(0) \int_0^\tau\left( \dot u \dot q - a^+ u^3 q\right) +
u(\tau) \int_0^\tau\left( \dot u \dot p - a^+ u^3 p\right).
$$
Since, for $\eps \to 0^+$,
$$
\left\vert \int_0^\tau \left( \dot u^2 - a^+ u^4\right)\right\vert +\vert \chi_2(u)\vert \to 0
$$
uniformly in the class of functions satisfying \eqref{ipotesilem1},
we have that \eqref{formula1} follows from the definition of $\rho$ in \eqref{sceltarho}, and \eqref{derivatelambda}
(the estimate for the derivative w.r.t. $\beta$ can be obtained in the same way).

With even simpler computations we find
\begin{align*}
\frac{\partial}{\partial\alpha'}J_{[0,\tau]}(U)\vert_{P_0} & = \partial_{\alpha'} \lambda(P_0) \int_0^\tau \left( \dot V \partial_\lambda \dot V - a^+ V^3 \partial_{\alpha} V_\lambda\right)_{P_0}\\
& = \partial_{\alpha'} \lambda(P_0) \left[ \int_0^\tau \left( \dot u^2 - a^+ u^4\right) - \chi_2(u)\right],
\end{align*}
so that, using similar arguments as above,
\eqref{formula3} follows.
\end{proof}

\subsection{An auxiliary boundary value problem}\label{subs:-}

In this section we establish some auxiliary results
dealing with the following boundary value problem set on $[-T+\tau,T]=I^-_{-1}\cup I^+_{0}\cup I^-_{0}$
\begin{equation}\label{connection}
\left\{
\begin{array}{l}
\vspace{0.2cm}
\ddot u + a_\mu(t)u^3 = 0 \\
\vspace{0.2cm}
u(\tau - T) = x, \; u(T) =  y \\
\int_{0}^{\tau} \dot u^2 < r^2,
\end{array}
\right.
\end{equation}
where $r > 0$ is as in \eqref{eq:r}.
We postpone more comments about the role
of the above written problem \eqref{connection} in connection
with the manifold $\mathcal{N}_\mu$ at the end of the Section
(see Remark \ref{remarkone}).

\begin{proposition}\label{propo:auxiliary}
For every $K> 0$, there exists $\mu^* > 0$ such that,
for any $\mu > \mu^*$ and $x,y \in [-K,K]$, problem \eqref{connection} admits a
unique solution $\bar{u}_\mu(t;x,y)$. Moreover,
\begin{itemize}
\item[(i)] $\Vert \bar{u}_\mu \Vert_{L^{\infty}(\tau-T,T)} \leq K$,
\item[(ii)] uniformly in $x,y \in [-K,K]$,
\begin{equation}\label{limitebaru}
\lim_{\mu \to +\infty}
\left(\Vert \bar{u}_\mu \Vert_{L^{\infty}(0,\tau)} +
\Vert \dot{\bar{u}}_\mu \Vert_{L^{\infty}(0,\tau)} \right) = 0,
\end{equation}
\item[(iii)] if $\bar{u}_\mu \not\equiv 0$ and
\begin{equation}\label{rotazioni}
\int_{\tau-T}^T \dot{\bar{u}}_\mu^2 < \left((2T-\tau)^3 \Vert a^+ \Vert_{L^\infty}\right)^{-1},
\end{equation}
then $\bar{u}_\mu$ cannot have neither more than one zero nor both a zero and a zero of the derivative in $[\tau-T,T]$.
As a consequence: if $xy>0$, then $\bar{u}_\mu$ has constant sign on $[\tau-T,T]$; if $xy\leq 0$,
then $\bar{u}_\mu$ vanishes exactly once in $[\tau-T,T]$ and $\dot{\bar{u}}_\mu(t) \neq 0$
for any $t \in [\tau-T,T]$.
\end{itemize}
\end{proposition}

\begin{proof}
We start by recalling the following:
\emph{for every $K > 0$ and $\eps > 0$, there exists $\mu^* > 0$ such that,
for any $\mu > \mu^*$ and $x,y \in [-K,K]$, then any solution $u$ of \eqref{connection} satisfies}
\begin{equation}\label{quote}
\Vert u \Vert_{L^{\infty}(0,\tau)} +
\Vert \dot{u} \Vert_{L^{\infty}(0,\tau)} \leq \eps.
\end{equation}
This has already been proved in Lemma \ref{nehariprimeprop2} (ii)
and has many consequences. First of all, it implies \eqref{limitebaru}; from this, point (i) follows using convexity arguments
(see \eqref{convess}).
Second, it implies (iii). Indeed, if we assume that
$t_1,t_2 \in [\tau-T,T]$, $t_1<t_2$, are instants satisfying one of the above properties, multiplying the equation
$\ddot{u} +a_\mu(t)u^3 = 0$ by $u$ and integrating by parts on $[t_1,t_2]$ we find
$$
\int_{t_1}^{t_2} \dot u^2 = \int_{t_1}^{t_2} a_\mu u^4 \leq (2T-\tau) \Vert a^+ \Vert_{L^\infty} \Vert u \Vert^4_{L^{\infty}(t_1,t_2)}.
$$
Using the Sobolev inequality \eqref{sobolev} we obtain
$$
\int_{t_1}^{t_2} \dot u^2 \leq (2T-\tau)^3 \Vert a^+ \Vert_{L^\infty} \left(\int_{t_1}^{t_2} \dot u^2\right)^2,
$$
contradicting \eqref{rotazioni}.

Finally, also the uniqueness of the solution to \eqref{connection} follows from \eqref{quote}.
Indeed, assume that $u_i$, $i=1,2$, are two different solutions of
\eqref{connection} and define $w =u_2-u_1$; of course $w$ satisfies
\[
\begin{cases}
\ddot{w}+a_\mu(t) (u_2^2+u_1u_2+u_1^2)w =0\\
w(\tau-T) = w(T) = 0,
\end{cases}
\]
Multiplying the previous equation by $w^-$, after an integration on $[\tau-T,T]$ we obtain
\[
\begin{split}
\int_{\tau-T}^{T} (\dot w^-)^2
& = \int_{\tau-T}^{T} a_\mu(u_2^2+u_1 u_2 + u_1^2)(w^-)^2 \\
& \leq \int_{[0,\tau] \cap \mathrm{supp}(w^-)} a^+  (u_2^2+u_1 u_2 + u_1^2)(w^-)^2;
\end{split}
\]
from \eqref{quote} and the Poincar\'e inequality \eqref{poincare}
\[
\int_{\tau-T}^{T} (\dot w^-)^2
\leq \eps^2 (2T-\tau)^2 \Vert a^+ \Vert_{L^\infty} \int_{\tau-T}^{T} (\dot w^-)^2.
\]
As $\mu$ is sufficiently large we obtain that necessarily $w^- \equiv 0$, that is $u_2-u_1\geq0$. Exchanging the roles of $u_1$ and $u_2$,  we infer $u_1 \equiv u_2$.

\medbreak
We now turn to the existence, which is obtained via a variational argument
and is organized in three steps.
\smallbreak
\noindent
\emph{\underline{Step 1}. Reduction to a constrained minimization problem.}
\smallbreak
\noindent
Consider the minimization problem
\begin{equation}\label{probmin}
\inf_{u \in \mathcal{M}(x,y)}  \widehat{J}_{\mu}(u)
\end{equation}
where
$$
\widehat{J}_{\mu}(u) = \frac{1}{2}\int_{\tau-T}^T \dot u^2 - \frac{1}{4}\int_{\tau-T}^T a_\mu u^4, \qquad
u \in H^1(\tau - T,T),
$$
and
$$
\mathcal{M}(x,y) = \left\{ u \in H^1(\tau - T,T) \; \Big \vert \;
\begin{array}{l}
u(\tau-T) = x, \, u(T) =  y
\smallskip\\
\vert u(0) \vert \leq K
\smallskip\\
\int_{0}^{\tau} \dot u^2 \leq r^2
\end{array}
\right\}.
$$
We observe that, for any $\mu > 0$, such a minimization problem is for sure solvable.
Indeed, the constraints $\vert u(0) \vert \leq K$ and $\int_{0}^{\tau} \dot u^2 \leq r^2$
imply that
\begin{equation}\label{coerc}
\Vert u \Vert_{L^\infty(0,\tau)} \leq K + \sqrt{\tau} r, \quad \mbox{ for every } u \in \mathcal{M}(x,y).
\end{equation}
Hence, $\widehat J_\mu$ is coercive (and weakly lower semicontinuous) on the closed convex set
$\mathcal{M}(x,y)$ and the conclusion follows from classical arguments.

Of course (by writing the Euler-Lagrange equation associated with
\eqref{probmin}) we can prove that a solution of \eqref{probmin} solves \eqref{connection} whenever
$\vert u(0) \vert < K$ and
$\int_{0}^{\tau} \dot u^2 < r^2$. Steps 2 and 3 below will then be devoted to prove that
this is the case when $\mu$ is large enough. In the following, we use the notation
$$
K_{[t_1,t_2]}(u) =  \frac{1}{2}\int_{t_1}^{t_2} \dot u^2, \qquad
U_{\mu,[t_1,t_2]}(u) = \frac{1}{4}\int_{t_1}^{t_2} a_\mu u^4
$$
and
$$
\widehat{J}_{\mu,[t_1,t_2]}(u) = K_{[t_1,t_2]}(u) - U_{\mu,[t_1,t_2]}(u)
$$
for a subinterval $[t_1,t_2] \subset [\tau-T,T]$.

As a preliminary observation for the arguments below, we also notice that
every minimizer of \eqref{probmin} solves the equation
$\ddot u -\mu a^-(t)u^3 = 0$ on the interval $[\tau-T,0] \cup [\tau,T]$.
Indeed, variations vanishing on $[0,\tau]$ are admissible
for functions in $\mathcal{M}(x,y)$.
\smallbreak
\label{vincolo1}
\noindent
\emph{\underline{Step 2}. For $\mu$ large, the minimizers of \eqref{probmin} satisfy $\vert u(0) \vert < K$.}
\smallbreak
\noindent
Let $u = u_\mu$ be a minimizer of \eqref{probmin}
and assume by contradiction
that $\vert u(0) \vert = K$.
Our aim is to construct a function
$\widetilde{u} \in \mathcal{M}(x,y)$ such that
\begin{equation}\label{contr}
\widehat{J}_\mu(\widetilde{u}) < \widehat{J}_\mu(u),
\end{equation}
and thus contradicting the minimality of $u$.

To this end, we first construct $\widetilde{u}$ on $[\tau-T,0]$.
Let $\delta >0$ be small; according to Lemma \ref{decay},
we have that $\vert u(t) \vert < K/4$ for every $t \in [\tau-T+\delta,-\delta]$,
provided $\mu$ is chosen so large that
\begin{equation}\label{sceltamu}
C_\delta \left( \frac{K}{\mu} \right)^{1/3} < \frac{K}{4},
\end{equation}
with $C_\delta$ the constant appearing in \eqref{paregiusto}.
Then, we choose three points $t_1,t_2$ and $t^*$ in
$[\tau-T,0]$ as follows: $t_1$ is the minimum of
$\vert u \vert$ (hence $\vert u(t_1) \vert < K/4$);
$t_2 \geq t_1$ is the unique point such that
$$
u(t_2) = 2 u(t_1)
$$
(hence $\vert u(t_2) \vert < K/2$);
$t^* > t_2$ is such that
\begin{equation}\label{t*}
\vert u(t^*) \vert = \frac{K}{2}.
\end{equation}
Notice that $t^* \geq -\delta$.
Set now
$$
\widetilde{u}(t) =
\left\{
\begin{array}{ll}
\vspace{0.2cm}
u(t) & \; \mbox{ if } t \in [\tau-T,t_1] \\
\vspace{0.2cm}
2 u(t_1) - u(t) & \; \mbox{ if } t \in [t_1,t_2] \\
0 & \; \mbox{ if } t \in [t_2,0]
\end{array}
\right.
$$
(of course, the interval $[\tau-T,t_1]$
is empty if the minimum of $\vert u \vert$ is achieved at $\tau-T$, while
the interval $[t_1,t_2]$ is empty if $u$ changes sign on $[\tau-T,0]$,
namely $u(t_1) = 0$). Notice that:
\begin{itemize}
\item on $[\tau - T,t_1]$, $u = \widetilde{u}$, so that
$\widehat{J}_{\mu,[\tau-T,t_1]}(u) = \widehat{J}_{\mu,[\tau-T,t_1]}(\widetilde u)$;
\item on $[t_1,t_2]$, $\vert \dot u \vert = \vert \dot{\widetilde{u}}\vert$ and
$\vert u \vert \geq \vert\widetilde{u} \vert$, so that
$K_{[t_1,t_2]}(u) = K_{[t_1,t_2]}(\widetilde u)$ and
$U_{\mu,[t_1,t_2]}(u) \leq U_{\mu,[t_1,t_2]}(\widetilde u)$;
hence $\widehat{J}_{\mu,[t_1,t_2]}(u) \geq \widehat{J}_{\mu,[t_1,t_2]}(\widetilde u)$
\item on $[t_2,0]$, $\widehat{J}_{\mu,[t_2,0]}(\widetilde u) = 0$ while (since $K/2 = |u(0)|-|u(t^*)| \geq \int_{t^*}^0 |\dot{u}|$)
\begin{equation}\label{crucial}
\widehat{J}_{\mu,[t_2,0]}(u) \geq \frac{1}{2}\int_{t^*}^{0} \dot u^2 \geq \frac{K^2}{8}\vert t^* \vert^{-1}
\geq \frac{K^2}{8}\delta^{-1}.
\end{equation}
\end{itemize}
To construct $\widetilde{u}$ on $[0,T]$, we distinguish
two possibilities. Set
$$
K' = r\sqrt{\tau}.
$$
If $\vert u(\tau) \vert < K'$, we define
$$
\widetilde{u}(t) =
\left\{
\begin{array}{ll}
\vspace{0.2cm}
u(\tau)\frac{t}{\tau} & \; \mbox{ if } t \in [0,\tau] \\
u(t) & \; \mbox{ if } t \in [\tau,T].
\end{array}
\right.
$$
Notice that
$$
\int_{0}^{\tau} \dot{\widetilde{u}}^2 < \frac{(K')^2}{\tau} = r^2
$$
so that $\widetilde{u} \in \mathcal{M}(x,y)$. If $\vert u(\tau) \vert \geq K'$, we set
$\widetilde{u}(t) = 0$ for $t \in [0,\tau]$
and then define $\widetilde{u}(t)$ on $[\tau,T]$
in a similar way to what has been done on $[\tau-T,0]$
(with $K$ replaced by $K'$).
In any case, we have that:
\begin{itemize}
\item $\widehat{J}_{\mu,[\tau,T]}(u) \geq \widehat{J}_{\mu,[\tau,T]}(\widetilde u)$;
\item in view of \eqref{coerc}, both $\widehat{J}_{\mu,[0,\tau]}(u)$
and $\widehat{J}_{\mu,[0,\tau]}(\widetilde u)$
are bounded below by a constant independent of $u$ and $\mu$.
\end{itemize}
At this point, it is enough to observe that
for $\delta \to 0^+$ the term in \eqref{crucial} goes to infinity.
In view of \eqref{sceltamu}, we thus have a contradiction for $\mu$ large enough.
\smallbreak
\label{vincolo2}
\noindent
\emph{\underline{Step 3}. For $\mu$ large, the minimizers of \eqref{probmin} satisfy $\int_{0}^{\tau} \dot u^2 < r^2$.}
\smallbreak
\noindent
Before starting the proof, we observe the following: for every $\eps > 0$, if $\mu$ is large enough
any minimizer of \eqref{probmin} satisfies $\vert u(0) \vert, \vert u(\tau) \vert \leq \eps$.
This can be shown using the very same arguments of the previous Step 2.
Hence, arguing as in the proof of Lemma \ref{nehari1} - and recalling the choice of $r$ in \eqref{eq:r} - of we can see that the following relation holds
\begin{equation}\label{relazione}
\int_0^\tau a^+ u^4 \leq K_\eps + \frac{1}{4 r^2} \left( \int_{0}^{\tau} \dot u^2 \right)^2,
\end{equation}
where $K_\eps = 8 \Vert a^+ \Vert_{L^\infty} \tau \eps^4$.

Now, let $u = u_\mu$ be a minimizer of \eqref{probmin} and assume by contradiction that
$\int_{0}^{\tau} \dot u^2 = r^2$; as in Step 2, we aim at constructing $\widetilde{u} \in \mathcal{M}(x,y)$
satisfying \eqref{contr}. Here we simply define $\widetilde{u}(t) = u(t)$ for $t \notin [0,\tau]$
and $\widetilde{u}|_{[0,\tau]}$ as the affine functions for $(0,u(0))$ and $(\tau,u(\tau))$.
Since $\vert u(0) \vert, \vert u(\tau) \vert \leq \eps$, we see that $\widetilde{u} \in \mathcal{M}(x,y)$
for $\eps$ small enough and
\begin{equation}\label{rel1}
\widehat{J}_{\mu,[0,\tau]}(\widetilde u) = o(1), \quad \mbox{ for } \, \eps \to 0^+.
\end{equation}
On the other hand, taking into account \eqref{relazione} we have
\begin{equation}\label{rel2}
\widehat{J}_{\mu,[0,\tau]}(u) \geq \frac{r^2}{2}\left( 1 - \frac{1}{8}\right) - \frac{K_\eps}{4} = \frac{7}{14}r^2 - \frac{K_\eps}{4} > 0
\end{equation}
for $\eps$ small enough. Combining \eqref{rel1} and \eqref{rel2}, we have that \eqref{contr}
holds true for
$\mu$ large enough, as desired.
\end{proof}

In the next proposition we collect some useful properties of the solution
$\bar u_\mu(t;x,y)$.

\begin{proposition}\label{propsoluzione}
Under the assumptions of Proposition \ref{propo:auxiliary},
the map $(x,y) \mapsto \bar u_\mu(t;x,y)$ $\in W^{2,\infty}(\tau-T,T)$ is of class $C^1$;
setting
$$
v_\mu(t;x,y) = \frac{\partial}{\partial x} \bar u_\mu(t;x,y),
\qquad
z_\mu(t;x,y) = \frac{\partial}{\partial y} \bar u_\mu(t;x,y),
$$
the following hold true (up to enlarging $\mu^*$ if necessary):
\begin{itemize}
\item[$(i)$] $v_\mu$, $z_\mu$ solve, respectively, the linear boundary value problems
\begin{equation}\label{bvp}
\left\{
\begin{array}{l}
\vspace{0.2cm}
\ddot{v} + 3 a_\mu(t) \bar u_\mu^2(t;x,y) v  = 0 \\
v(\tau-T) = 1, \; v(T) =  0
\end{array}
\right.
\qquad
\left\{
\begin{array}{l}
\vspace{0.2cm}
\ddot{z} + 3 a_\mu(t) \bar u_\mu^2(t;x,y) z  = 0 \\
z(\tau-T) = 0, \; z(T) =  1,
\end{array}
\right.
\end{equation}
and, moreover, $v_\mu$ is positive and decreasing and
$z_\mu$ is positive and increasing;
\item[$(ii)$] it holds
\begin{equation}\label{derliv1}
\frac{\partial }{\partial x} J_\mu(\bar u_\mu(\cdot;x,y)) = - \dot{\bar{u}}_\mu((\tau-T)^+;x,y)
\end{equation}
and
$$
\frac{\partial }{\partial y} J_\mu(\bar u_\mu(\cdot;x,y)) = \dot{\bar{u}}_\mu(T^-;x,y);
$$
\item[$(iii)$] it holds
\begin{equation}\label{derivata1}
\vert \dot v_\mu(T^-;x,y) \vert, \, \vert \dot z_\mu((\tau-T)^+;x,y) \vert \leq \frac{2}{2T-\tau}
\end{equation}
and
\begin{equation}\label{derivata2}
\vert x \dot v_\mu((\tau-T)^+;x,y) \vert, \, \vert y \dot z_\mu(T^-;x,y) \vert
\leq \frac{2K}{2T-\tau} + 5 \Vert \dot{\bar{u}}_\mu \Vert_{L^{\infty}(T-\tau,T)};
\end{equation}
moreover, if $\max(\vert \dot{\bar{u}}_\mu((\tau-T)^+;x,y) \vert, \vert \dot{\bar{u}}_\mu(T^-;x,y)\vert) \leq \rho$,
then
\begin{equation}\label{derivata2b}
\vert x \dot v_\mu((\tau-T)^+;x,y) \vert, \, \vert y \dot z_\mu(T^-;x,y) \vert
\leq \frac{2K}{2T-\tau} + 5 \rho;
\end{equation}
\item[$(iv)$] if $\bar{u}_\mu$ has constant sign on $[\tau-T,T]$,
and
$$
- \dot{\bar{u}}_\mu((\tau-T)^+;x,y) = \dot{\bar{u}}_\mu(T^-;x,y) = \rho \quad \mbox{ if } \, \bar{u}_\mu > 0
$$
$$
\dot{\bar{u}}_\mu((\tau-T)^+;x,y) = - \dot{\bar{u}}_\mu(T^-;x,y) = \rho \quad \mbox{ if } \, \bar{u}_\mu < 0,
$$
then it holds
\begin{equation}\label{derivata3}
\dot v_\mu((\tau-T)^+;x,y) + \dot z_\mu((\tau-T)^+;x,y) < 0 < \dot v_\mu(T^-;x,y) + \dot z_\mu(T^-;x,y).
\end{equation}
\end{itemize}
In points (iii) and (iv) of the above statement, $\rho$ is as in \eqref{sceltarho}.
\end{proposition}

\begin{proof}
As already proved, $\bar u_\mu$ is unique; in the same spirit, one can use property
\eqref{quote}, which says that the minimization problem is set in an almost convex case,
to show that $\bar u_\mu$ is also non degenerate (i.e. that the linearized equation, with
homogeneous Dirichlet boundary conditions, has no nontrivial solution). Using this and the
Fredholm's Alternative, the $C^1$-dependence
from $x$ and $y$ follows in a standard way (see \cite[Section 5]{OrtVer04} for the full details
in a similar situation).

Now we prove separately each point of the statement; in points
(i),(ii),(iii) we concentrate our attention on $v_\mu$, the proof for $z_\mu$ being analogous
(for simplicity of notation, we omit the dependence on $x$ and $y$ if no confusion is possible).
\smallbreak
The fact that $v_\mu$ solves the boundary value problem in \eqref{bvp} is well known;
we prove that $v_\mu$ is positive and decreasing.
Multiplying the equation in \eqref{bvp} by $v_\mu$ and integrating by parts, we find
for any $t \in [\tau-T,T[$
$$
v_\mu(t)\dot v_\mu(t) = -\int_t^T \dot v_\mu^2 + 3 \int_t^T a_\mu \bar{u}_\mu^2 v^2
\leq -\int_t^T \dot v_\mu^2 + 3 \Vert a^+ \Vert_{L^{\infty}} \Vert \bar{u}_\mu \Vert^2_{L^{\infty}(0,\tau)} \int_t^T v_\mu^2.
$$
By \eqref{poincare} we have
$$
v_\mu(t)\dot v_\mu(t) \leq \left( -1 + 3 (2T-\tau)^2 \Vert a^+ \Vert_{L^{\infty}} \Vert \bar{u}_\mu \Vert^2_{L^{\infty}(0,\tau)}\right) \int_t^T \dot v_\mu^2.
$$
Recalling \eqref{limitebaru}, we finally obtain, for $\mu$ large,
$$
v_\mu(t)\dot v_\mu(t) \leq -\frac{1}{2} \int_t^T \dot v_\mu^2 < 0,
\quad \mbox{ for every } t \in [\tau-T,T],
$$
and this proves our claim.
\smallbreak
To prove \eqref{derliv1}, we simply observe that - integrating by parts -
$$
\frac{\partial }{\partial x} J_\mu(\bar u_\mu) = \int_{\tau-T}^T \left( \dot{\bar u}_\mu \dot v_\mu - a_\mu \bar u_\mu^3 v_\mu \right)
= [\dot{\bar u}_\mu v_\mu]_{\tau-T}^T - \int_{\tau-T}^T \left( \ddot{\bar u}_\mu  + a_\mu \bar u_\mu^3 \right) v_\mu,
$$
whence the conclusion.
\smallbreak
As for \eqref{derivata1}, we first use the mean value theorem
to find $\bar{t} \in [\tau-T,T]$ such that
$\dot v_\mu(\bar{t}) = -1/(2T-\tau)$. From this, and using the fact that $0 \leq v_\mu \leq 1$,
\begin{align*}
\dot v_\mu(T) & = \dot v_\mu(\bar{t}) + \int_{\bar{t}}^T \ddot v_\mu =
- \frac{1}{2T-\tau} - 3\int_{\bar{t}}^T a_\mu\bar{u}_\mu^2v_\mu \\
& \geq - \frac{1}{2T-\tau}- 3 \Vert a^+ \Vert_{L^{\infty}} \Vert \bar{u}_\mu \Vert^2_{L^{\infty}(0,\tau)}.
\end{align*}
Recalling \eqref{limitebaru}, we conclude
for $\mu$ large. As for \eqref{derivata2}, taking into account the equations for $\bar{u}_\mu$ and  $v_\mu$,
we obtain
$$
3\ddot{\bar{u}}_\mu(t)v_\mu(t)-\bar{u}_\mu(t) \ddot v_\mu(t) = 0.
$$
Integrating by parts we have
$$
- 3\dot{\bar u}_\mu(\tau-T) + x \dot v_\mu(\tau-T) - y \dot v_\mu(T)  - 2\int_{\tau-T}^{T}
\dot{\bar{u}}_\mu \dot v_\mu = 0.
$$
Now, the term $y \dot v_\mu(T)$ can be estimated using \eqref{derivata1}, while for the integral we have
$$
\left\vert\int_{\tau-T}^{T}\dot{\bar{u}}_\mu(t)\dot v_\mu(t)\right\vert \leq
- \Vert \dot{\bar{u}}_\mu \Vert_{L^{\infty}(\tau-T,T)}\int_{\tau-T}^{T} \dot v_\mu = \Vert \dot{\bar{u}}_\mu \Vert_{L^{\infty}(\tau-T,T)}.
$$
This proves \eqref{derivata2}. As for \eqref{derivata2b}, we just need to observe that
$$
\max(\vert \dot{\bar{u}}_\mu(\tau-T;x,y) \vert, \vert \dot{\bar{u}}_\mu(T;x,y)\vert) \leq \rho
\quad \Longrightarrow \quad \Vert \dot{\bar u}_\mu \Vert_{L^\infty(\tau-T,T)} \leq \rho.
$$
This follows from \eqref{limitebaru}, together with convexity arguments
(see \eqref{convess_derivata}).
\smallbreak
Finally, we prove \eqref{derivata3}; we concentrate on the case $\bar u_\mu > 0$, the other being analogous.
Set $w_\mu = v_\mu + z_\mu$; then
$w_\mu$ solves the boundary value problem
$$
\left\{
\begin{array}{l}
\vspace{0.2cm}
\ddot{w} + 3 a_\mu(t) \bar u_\mu^2(t;x,y) w  = 0 \\
w(\tau-T) = w(T) = 1,
\end{array}
\right.
$$
it holds $0 \leq w_\mu \leq 2$
and our thesis reads as $\dot w_\mu(\tau-T) < 0 < \dot w_\mu(T)$.
Let us assume by contradiction that
$\dot w_\mu(\tau-T) \geq 0$ (the argument for $\dot w_\mu(T)$ being the same);
then, usual convexity arguments imply that
$w_\mu(t) \geq 1$ for $t \in [\tau-T,0]$ and that
there exists $t^* \in [0,\tau]$ such that
\[
\dot w_\mu(t^*) = 0
\]
(indeed, $\dot w_\mu(0)>0$ and $w_\mu(0)>1$; hence, if $\dot w_\mu(t)>0$ for every $t \in [0,\tau]$, then also $w_\mu(\tau)>1$ and, by convexity, $w_\mu(T)>1$, a contradiction).

Now, by Proposition \ref{propo:auxiliary}, (i), $\bar u_\mu\leq K$ on $[\tau-T,0]$; being $w_\mu\geq 1$ on the same interval, we deduce that $3 \ddot{\bar{u}}_\mu \leq K \ddot w_\mu$. Hence, for every $t \in [\tau-T,0]$,
\[
\begin{split}
\dot{\bar u}_\mu(t) &= \dot{\bar u}_\mu(\tau -T) + \int_{\tau-T}^t \ddot{\bar u}_\mu
\leq -\rho + \frac{K}{3} \int_{\tau-T}^t \ddot w_\mu \\
&= - \rho + \frac{K}{3} \left( \dot w_\mu(t) - \dot w_\mu(\tau-T)\right) \leq - \rho + \frac{K}{3} \dot w_\mu(t) = -\rho - \frac{K}{3} \int_{t}^{t^*} \ddot w_\mu \\
&\leq -\rho + K \int_t^{t^*} a_\mu \bar u_\mu^2 w_\mu \leq -\rho + 2K  \Vert a^+ \Vert_{L^{\infty}} \Vert \bar{u}_\mu \Vert^2_{L^{\infty}(0,\tau)}.
\end{split}
\]
In particular, according to \eqref{limitebaru} we can assume that
$\dot{\bar u}_\mu(t) \leq -\rho/2$ for every $t \in [\tau-T,0]$
provided $\mu$ is large enough. Hence (recall \eqref{sceltarho})
$$
\bar u_\mu(0) \leq K - \frac{\rho}{2}(T-\tau) \leq 0,
$$
contradicting the fact that $\bar u_\mu > 0$.
\end{proof}

We conclude this section with an important remark.

\begin{remark}\label{remarkone}
We observe at first that an existence and uniqueness result holds true,
for any $\mu > 0$, for the boundary value problem
\begin{equation}\label{connectionr}
\left\{
\begin{array}{l}
\vspace{0.2cm}
\ddot u - \mu a^-(t)u^3 = 0 \\
u(\tau) = x, \; u(T) =  y;
\end{array}
\right.
\end{equation}
moreover (by convexity arguments), nontrivial solutions to \eqref{connectionr}
are of constant sign if $xy > 0$, and are strictly monotone and vanish exactly once if
$xy \leq 0$. The proof of these facts is straightforward, due to the coercivity and convexity of the action
functional associated with \eqref{connectionr}. Moreover,
a result completely analogous to Proposition \ref{propsoluzione} can be proved
for \eqref{connectionr}
with simpler arguments (using, again, the convexity properties of the solutions).
From this point view, Propositions \ref{propo:auxiliary}
and \ref{propsoluzione} can be viewed as an extension of the more elementary results holding
for \eqref{connectionr} to the boundary value problem \eqref{connection}.
Roughly speaking, it can be said that -
as the integral bound $\int_0^\tau \dot u^2 < r^2$ is considered - problem \eqref{connection}
behaves exactly like \eqref{connectionr} when $\mu$ is large (notice, indeed, that - by point (ii)
of Proposition \ref{propo:auxiliary} - solutions to \eqref{connection} are arbitrarily small
on $[0,\tau]$ when $\mu \to +\infty$).

We now claim that Propositions \ref{propo:auxiliary} and \ref{propsoluzione}
can be generalized for the problem
\begin{equation}\label{connectionr2}
\left\{
\begin{array}{l}
\vspace{0.2cm}
\ddot u + a_\mu(t)u^3 = 0 \\
\vspace{0.2cm}
u(\tau_i) = x, \; u(\sigma_{i+l+1}) =  y \\
\int_{\sigma_j}^{\tau_j} \dot u^2 < r^2 \quad \mbox{ for } j = i+1,\ldots,i+l,
\end{array}
\right.
\end{equation}
(of course, \eqref{connection} corresponds to the case $i=-1$ and $l=1$).
The proofs can be obtained following exactly the same arguments given before,
at the only expense of an unpleasant
overloading of the notation. An important warning, however, is that in such a case the value
$\mu^*$, as well as the constants in \eqref{derivata1}, \eqref{derivata2} and \eqref{derivata2b},
are depending on the integer $l$, that is, on the number of intervals of positivity
of the weight function $a_\mu$ in the interval $[\tau_i,\sigma_{i+l+1}]$.
Accordingly, all these constants can me made uniform
for all problems of the type \eqref{connectionr2} with $l$ less than or equal to
a common bound $k$ (compare with assumption \eqref{ipotesil}).

It now should be clear that Proposition \ref{propo:auxiliary}
can be viewed as a tool for characterizing (and constructing) functions in $\mathcal{N}_\mu$
on intervals of the type
$$
I^-_{i} \cup I^+_{i+1} \cup I^-_{i+1} \cup \ldots \cup I^-_{i+l},
\quad \mbox{ with } i+1,\ldots,i+l \not\in L.
$$
By the assumption \eqref{ipotesil}, intervals of this type can exist only if $l \leq k$,
and this allows to determine a precise value $\mu^*$ such that the construction explained in this section is possible.
It is worth noticing that, if $\mu > \mu^*$ is fixed, solutions of the boundary value problem
\eqref{connectionr2} give rise to (a restriction of) a function in
$\mathcal{N}_\mu$ only if $x,y$ are small enough. Indeed, the further condition
(C4) has to be satisfied (see Lemma \ref{upiccola}).

In the next Section \ref{secexistence}, we will actually deal (for simplicity of notation) only with
problem \eqref{connection}, that is, we will assume that $k = 1$.
As remarked above, however, the general case $k \geq 1$ could be treated as well.
\end{remark}

\section{Construction of a constrained Palais-Smale sequence}\label{secexistence}

The aim of this section is to construct a (bounded) constrained Palais-Smale sequence $(u_n) \subset \mathcal{N}_\mu$
at level \eqref{consmin}. This will be done using Ekeland's variational principle
\cite[Chapter 4]{DeF89} in a careful way.
Again (compare with the discussion at the beginning of Section
\ref{secproofnehari}) this construction will possibly require to enlarge $\mu^*$ more and more times,
according to the estimates collected in the previous sections (and ultimately depending on the local estimate
contained in Lemma \ref{upiccola}). In this way, we can still produce a threshold $\mu^*$
depending on the weight function $a$ and on the integer $k$,
but not on the sequence $\mathcal{L}$ and the integer $N$.
\medbreak
We now start with our arguments. The Ekeland's variational principle applied to $\overline{\mathcal{N}}_\mu$
yields the existence of a minimizing sequence $(u_n) \subset \overline{\mathcal{N}}_\mu$ for problem \eqref{consmin}
such that
\begin{equation}\label{ekeland}
J_\mu(u_n) \leq J_\mu(u) + \frac{1}{n}\Vert u - u_n \Vert_{H^1(I_N)}, \qquad \mbox{ for every } u \in \overline{\mathcal{N}}_\mu,\, n \in \mathbb{N}.
\end{equation}
Our goal now is to show that, for $n$ large enough, $u_n \notin \overline{\mathcal{N}_\mu}\setminus
\mathcal{N}_\mu$.
Indeed, if this is the case, then \eqref{ekeland} implies in a standard way
that the constrained gradient $\nabla_{\mathcal{N}_\mu} J (u_n) \to 0$, as required.

We first observe that, in view of Lemma \ref{nehariprimeprop2}, the only possibilities
for $u_n$ to be on $\partial \mathcal{N}_\mu$ are given by
\begin{equation}\label{bordo1}
\int_{I^+_i} \dot u_n^2 = 2(c + c_\zeta), \qquad \mbox{ for some } i \in L
\end{equation}
or, for some $i \in L$,
\begin{equation}\label{bordo2}
\left\{
\begin{array}{l}
\dot u_n(\sigma_i^-) = \rho \\
u_n(\sigma_i) \geq 0,
\end{array}
\right.
\; \mbox{ or } \;
\left\{
\begin{array}{l}
\dot u_n(\sigma_i^-) = -\rho \\
u_n(\sigma_i) \leq 0,
\end{array}
\right.
\; \mbox{ or } \;
\left\{
\begin{array}{l}
\dot u_n(\tau_i^+) = -\rho \\
u_n(\tau_i) \geq 0,
\end{array}
\right.
\; \mbox{ or } \;
\left\{
\begin{array}{l}
\dot u_n(\tau_i^+) = \rho \\
u_n(\tau_i) \leq 0.
\end{array}
\right.
\end{equation}
The rest of this section will be devoted to exclude both the possibilities.
\medbreak
\noindent
\emph{\underline{Claim 1}. The possibility \eqref{bordo1} cannot occur, for $n$ large enough.}
\medbreak
\noindent
Indeed, assume that \eqref{bordo1} holds true for some $n$ and $i \in L$;
just for simplicity of notation, suppose moreover that $i=0$.
Take $p,q$ be as in \eqref{defpq}
and let $\bar{u} \in \mathcal{K}_0$, $\bar{u}>0$, according to \eqref{eq:intro_nehari}.
Then, define $F: \mathbb{R}^4 \to \mathbb{R}$ by setting
$$
F(\lambda,\alpha,\beta,\gamma) = \int_0^{\tau}  \left( \dot u^2(t;\lambda,\alpha,\beta) - a^+ u^4(t;\lambda,\alpha,\beta)\right)
- \gamma,
$$
where
$$
u(t;\lambda,\alpha,\beta) = \alpha q(t) +\beta p(t) + \lambda \bar{u}(t).
$$
Notice that $u(0;\lambda,\alpha,\beta) = \alpha$,
$u(\tau;\lambda,\alpha,\beta) = \beta$
and $u(t;1,0,0) = \bar{u}(t)$.
Since $F(1,0,0,0) = 0$ and
$$
\frac{\partial F}{\partial \lambda}(1,0,0,0) = -2 \int_{I^+_i}\dot{\bar{u}}^2 \neq 0,
$$
the implicit function theorem yields
the existence of a function $\lambda = \lambda(\alpha,\beta,\gamma)$,
defined a neighborhood of $(0,0,0)$ and having values in a neighborhood
of $1$, such that $F(\lambda(\alpha,\beta,\gamma),\alpha,\beta,\gamma) = 0$.

In view of Lemma \ref{nehariprimeprop},
we can assume, provided $\mu$ is large enough, that $\alpha_n= u_n(0)$, $\beta_n=u_n(\tau)$ and
$$
\gamma_n = u_n(\tau)\dot u_n(\tau^+)- u_n(0)\dot u_n(0^-)
$$
are as small as we wish. Then, we can define the function
$\widetilde{u}_n$ as $\widetilde{u}_n(t) = u_n(t)$
for $t \notin [0,\tau]$ and
$$
\widetilde{u}_n(t) = u(t;\lambda(\alpha_n,\beta_n,\gamma_n),\alpha_n,\beta_n,\gamma_n), \quad \mbox{ for } t \in [0,\tau].
$$
It is not difficult to see that
$\widetilde{u}_n \in \overline{\mathcal{N}}_\mu$. Moreover, since $\int_0^{\tau}\dot{\bar{u}}^2 = 4c$
and $\widetilde{u}_n$ is arbitrarily $H^1$-near to $\bar{u}$ for
$\mu$ large, we have that
$$
\epsilon_n = \frac{1}{4}\int_0^{\tau}\dot{\widetilde{u}}_n^2 - c
$$
can be made arbitrarily small. Hence
\begin{align*}
J_\mu(u_n) - J_\mu(\widetilde{u}_n) &=  \frac{1}{4}\int_{I^+_i} (\dot u_n^2
-\dot{\widetilde{u}}_n^2) = \frac{c + c_\zeta}{2} - (c +\epsilon_n)  \\
& \geq \frac{c_\zeta - c}{2} - \epsilon_n \geq \frac{c_\zeta - c}{4}.
\end{align*}
On the other hand, \eqref{ekeland} gives
$$
J_\mu(u_n) - J_\mu(\widetilde{u}_n) \leq \frac{1}{n} \Vert u_n - \widetilde{u}_n \Vert_{H^1(I_N)}
$$
and this is a contradiction for $n$ large enough
(of course, $\Vert u_n - \widetilde{u}_n \Vert_{H^1(I_N)}$ is bounded).
\begin{remark}\label{rem:property3}
The above argument holds also when $c_\zeta$ is replaced by any $c'>c$. Therefore, using
Lemma \ref{nehariprimeprop} and taking $\mu$ large enough (depending on $c'$), one can construct
minimizing sequences with the property that
\[
\int_{I^+_i} \dot u_n^2 < 2(c + c'), \qquad \mbox{ for every } i \in L.
\]
\end{remark}
\medbreak
\noindent
\emph{\underline{Claim 2}. The possibility \eqref{bordo2} cannot occur,
for $n$ large enough.}
\medbreak
\noindent
Indeed, assume that \eqref{bordo2} holds true for some $n$ and $i \in L$. We will give the details of the proof
in the case $u_n(\tau_i) \geq 0$ and $\dot u_n(\tau_i^+) = -\rho$ (the other can be treated with similar arguments);
moreover, for simplicity of notation we suppose that $i = -1$, that is,
we are dealing with the case
$$
u_n(\tau-T) \geq 0 \quad \mbox{ and } \quad \dot u_n((\tau-T)^+) = - \rho.
$$
According to Remark \ref{remarkone}, we finally assume that
$k = 1$ and, moreover, we deal with the most difficult case in which the first interval of positivity
on the right of $\tau-T$ is not in $L$, that is:
$0 \not\in L$ and $\pm1 \in L$. Summing up, we are led to construct a local variation
of $u_n$ in the interval
$$
I^+_{-1} \cup I^-_{-1} \cup I^+_0 \cup I^-_0 \cup I^+_1 = [-T,\tau + T].
$$
Roughly speaking, we will use the variation of Section \ref{subs:+} on $I^+_{-1}$ and $I^+_{1}$,
and the one of Section \ref{subs:-} on the interval $I^-_{-1} \cup I^+_0 \cup I^-_0$. We have to distinguish three situations. In the following, to simplify the notation, we set
$$
x_n = u_n(\tau-T) \quad \mbox{ and } \quad y_n = u_n(T).
$$
\smallbreak
If $\vert \dot u_n(T^-) \vert < \rho$, we argue as follows.
We define a family of functions $\widetilde{u}_n(\cdot,\xi)$, with $\xi \geq 0$,
by setting $\widetilde{u}_n(t,\xi) = u_n(t)$ for $t \notin [-T,\tau+T]$ and
\begin{itemize}
\item for $t \in [-T,\tau-T]$,
\begin{equation}\label{def1}
\widetilde{u}_n(t,\xi) = U\left(u_n(-T),
x_n-\xi,\dot u_n(-T^-),\dot{\bar{u}}_\mu((\tau-T)^+;x_n-\xi,y_n)\right)(t),
\end{equation}
\item for $t \in [\tau-T,T]$,
\begin{equation}\label{def2}
\widetilde{u}_n(t,\xi) = \bar{u}_\mu(t;x_n-\xi,y_n),
\end{equation}
\item for $t \in [T,\tau+T]$,
\begin{equation}\label{def3}
\widetilde{u}_n(t,\xi) = U
\left(y_n,u_n(\tau+T),\dot{\bar{u}}_\mu(T^-;x_n-\xi,y_n),
\dot u_n((\tau+T)^+)\right)(t).
\end{equation}
\end{itemize}
In formulas \eqref{def1} and \eqref{def3}, we have denoted (with some abuse of notation)
by $U$ a local variation, constructed as in Lemma \ref{localvar}, of $u_n$ on the interval
$[-T,\tau-T]$ and $[T,\tau+T]$, respectively; in \eqref{def2}, $\bar{u}_\mu$ is the function
constructed in Proposition \ref{propo:auxiliary}.
Notice that the map $\xi \mapsto \widetilde{u}_n(\cdot;\xi) \in H^1(I_N)$ is of class
$C^1$ and, of course, $\widetilde{u}_n(\cdot,0) = u_n$.

We claim that, for $\xi > 0$ small enough, the function
$\widetilde{u}_n(\cdot,\xi)$ lies in $\overline{\mathcal{N}}_\mu$, possibly touching
the constraint \eqref{bordo2} only outside the intervals we are considering.
To show this, the most delicate condition to be checked is the one concerning
$\dot{\widetilde{u}}_n((\tau-T)^+,\xi)$.
We thus observe that, in view of Proposition \ref{propsoluzione},
$$
\frac{d}{d\xi}\dot{\widetilde{u}}_n((\tau-T)^+,\xi)|_{\xi = 0} = -v_\mu((\tau-T)^+;x_n,y_n) > 0
$$
so that $\dot{\widetilde{u}}_n((\tau-T)^+,\xi) > \dot u_n((\tau-T)^+) = -\rho$
for small $\xi > 0$ as desired.

We now claim that
\begin{equation}\label{eke1}
J_\mu(\widetilde{u}_n(\cdot,\xi)) \leq J_\mu(u_n) - \left(\frac{\rho}{2}\right)\xi + o(\xi) \quad \mbox{ for } \xi \to 0^+.
\end{equation}
To prove this, we are going to show that
\begin{equation}\label{livello0}
\frac{d}{d\xi}J_\mu(\widetilde{u}_n(\cdot,\xi))|_{\xi = 0} \leq -\frac{\rho}{2};
\end{equation}
in the following, we use the notation
$J_{\mu,[t_1,t_2]}$ for the restriction of the action functional to
a subinterval $[t_1,t_2] \subset I_N$.
From \eqref{derliv1} of Proposition \ref{propsoluzione}, we have
\begin{equation}\label{livello1}
\frac{d}{d\xi}J_{\mu,[\tau-T,T]}(\widetilde{u}_n(\cdot,\xi))|_{\xi = 0}  =
\dot{\bar{u}}_\mu((\tau-T)^+;x_n,y_n) = \dot u_n((\tau-T)^+) = - \rho.
\end{equation}
On the other hand, setting $\beta(\xi) = x_n  + \xi$
and $\beta'(\xi) = \dot{\bar{u}}_\mu((\tau-T)^+;x_n-\xi,y_n)$,
we can compute
\begin{multline*}
\frac{d}{d\xi}J_{\mu,[-T,\tau-T]}(\widetilde{u}_n(\cdot,\xi))|_{\xi = 0}  =
\frac{\partial}{\partial\beta}J_{\mu,[-T,\tau-T]}(U) \frac{d\beta}{d\xi}
 \quad +  \frac{\partial}{\partial\beta'}J_{\mu,[-T,\tau-T]}(U) \frac{d\beta'}{d\xi} \\
 = \frac{\partial}{\partial\beta}J_{\mu,[-T,\tau-T]}(U)
 \quad -  \frac{\partial}{\partial\beta'}J_{\mu,[-T,\tau-T]}(U) \dot v_\mu((\tau-T)^+;x_n,y_n),
\end{multline*}
where we agree that the derivatives with respect to $\beta$ and $\beta'$ are evaluated at the
point $(u_n(-T),\dot u_n(-T),\beta(0),\beta'(0))$ while the
derivatives with respect to $\xi$ are evaluated at $\xi = 0$.
From \eqref{formula1} and \eqref{formula3} of Lemma \ref{localvar} we thus obtain
$$
\left\vert\frac{d}{d\xi}J_{\mu,[-T,\tau-T]}(\widetilde{u}_n(\cdot,\xi))|_{\xi = 0}\right\vert \leq
\frac{\rho}{8} - C_\eps x_n \dot v_\mu((\tau-T)^+;x_n,y_n).
$$
From \eqref{derivata3} of Lemma \ref{propsoluzione}
$$
\vert x_n  v_\mu((\tau-T)^+;x_n,y_n) \vert \leq \frac{2K}{2T-\tau} + 5\rho,
$$
so that, provided $\eps$ is small enough (which, in turns, requires maybe enlarging $\mu^*$)
$$
\left\vert\frac{d}{d\xi}J_{\mu,[-T,\tau-T]}(\widetilde{u}_n(\cdot,\xi))|_{\xi = 0}\right\vert \leq \frac{\rho}{4}.
$$
The arguments for $t \in [T,\tau+T]$ are even simpler, since the dependence of $U$
on $\xi$ is just via $\dot{\bar{u}}_\mu(T^-;x_n-\xi,y_n)$.
Arguing as above, we are thus led to consider the
term $C_\eps y_n v_\mu(T^-;x_n,y_n)$ which can be estimated using
\eqref{derivata1} of Lemma \ref{propsoluzione}. Hence
$$
\left\vert\frac{d}{d\xi}J_{\mu,[T,\tau+T]}(\widetilde{u}_n(\cdot,\xi))|_{\xi = 0} \right\vert \leq \frac{\rho}{4}
$$
so that, recalling \eqref{livello1}, \eqref{livello0} follows.

Finally, we verify that there exists a constant $S$, independent on $n$ (but depending on $\mu$, which however
is fixed in this argument), such that
\begin{equation}\label{eke2}
\Vert \widetilde{u}_n(\cdot,\xi) - u_n \Vert_{H^1(I_N)} \leq S \xi + o(\xi) \quad \mbox{ for } \xi \to 0^+.
\end{equation}
To show this, we just need to observe
that $\left\Vert \frac{d}{d\xi}\widetilde{u}_n(\cdot,\xi)|_{\xi = 0} \right\Vert_{H^1(I_N)}$
can be bounded independently on $n$
(we stress once more that the dependence on $\mu$ cannot be avoided, but it is irrelevant for
the arguments below, since throughout this section $\mu$ is fixed).
This is clear for $t \in [\tau-T,T]$
(since $x_n,y_n \in [-K,K]$) and essentially comes from
\eqref{formula0}
for $t \in [-T,\tau-T] \cup [T,\tau+T]$
(of course, to bound the derivative $d\widetilde{u}_n/d\xi$
one needs - as before - to estimate both $dU/d\beta$
and $d\beta/d\xi$; terms of the type $d\beta/d\xi$
are bounded since $x_n,y_n \in [-K,K]$).

We are now in position to conclude. Indeed, combining \eqref{eke1} and \eqref{eke2} yields
$$
J_\mu(\widetilde{u}_n(\cdot,\xi)) + \frac{1}{n}\Vert \widetilde{u}_n(\cdot,\xi) - u_n \Vert_{H^1(I_N)} \leq J_\mu(u_n) + \left(\frac{S}{n}- \frac{\rho}{2}\right)\xi + o(\xi)
$$
and this contradicts \eqref{ekeland} for $n$ large and fixed, and $\xi$ sufficiently small.
The proof is thus concluded in the case $\vert \dot u_n(T^-) \vert < \rho$.

\smallbreak
If $\dot u_n(T^-) = -\rho$ (notice that in this case it has to be
$u_n(T) < 0$) we argue exactly as before. The only difference is that showing that
$\widetilde{u}_n(\cdot,\xi) \in \overline{\mathcal{N}}_\mu$
now requires the further observation that
$$
\frac{d}{d\xi}\dot{\widetilde{u}}_n(T^-,\xi)|_{\xi = 0} = - v_\mu(T^-;x_n,y_n) > 0
$$
so that, for small $\xi > 0$, $\dot{\widetilde{u}}_n(T^-,\xi) >-\rho$ as well.
\smallbreak
If $u_n(T^-) = \rho$ (notice that in this case it has to be
$u_n(T) > 0$) a slightly different argument is needed,
since the same variation as before would lead to a function
$\widetilde{u}_n(\cdot,\xi) \notin \overline{\mathcal{N}}_\mu$.
Hence, we define here $\widetilde{u}_n(t,\xi) = u_n(t)$ for $t \notin [-T,\tau+T]$ and
\begin{itemize}
\item for $t \in [-T,\tau-T]$,
$$
\widetilde{u}_n(t,\xi) = U\left(u_n(-T),
x_n-\xi,\dot u_n(-T^-),\dot{\bar{u}}_\mu((\tau-T)^+;x_n-\xi,y_n-\xi)\right)(t),
$$
\item for $t \in [\tau-T,T]$,
$$
\widetilde{u}_n(t,\xi) = \bar{u}_\mu(t;x_n-\xi,y_n-\xi),
$$
\item for $t \in [T,\tau+T]$,
$$
\widetilde{u}_n(t,\xi) = U\left(y_n-\xi,
u_n(\tau+T),\dot{\bar{u}}_\mu(T^-;x_n-\xi,y_n-\xi),\dot u_n((\tau+T)^+)\right)(t).
$$
\end{itemize}
With this definition, we have that, for $\xi > 0$ small enough, the function
$\widetilde{u}_n(\cdot,\xi)$ lies in $\overline{\mathcal{N}}_\mu$.
Indeed, from \eqref{derivata3} of Lemma \ref{propsoluzione}
$$
\frac{d}{d\xi}\dot{\widetilde{u}}_n((\tau-T)^+,\xi)|_{\xi = 0} = -
\left(v_\mu((\tau-T)^+;x_n,y_n)+z_\mu((\tau-T)^+;x_n,y_n)\right) > 0
$$
so that $\dot{\widetilde{u}}_n((\tau-T)^+,\xi) > \dot u_n((\tau-T)^+) = -\rho$
(for small $\xi > 0$)
and a completely symmetric argument works for $\dot{\widetilde{u}}_n(T^-,\xi)$.

At this point, the arguments leading to \eqref{eke1} and \eqref{eke2}
are essentially the same as before up to minor modifications and will be omitted.

\section{Conclusion of the proofs
}\label{secfinale}

As for Theorem \ref{thper}, we need to prove the positivity of the solutions found, the $C^1$ bound
\eqref{eq:C1bounds},
and the properties (P1), (P2), (P3), concerning the behavior for $\mu \to +\infty$.
\medbreak
Estimate \eqref{eq:C1bounds} can be easily deduced from Lemma \ref{nehariprimeprop2} and
property (C4) in the definition of
$\mathcal{N}_\mu$ by observing that, because of the convexity/concavity
properties of a solution $u$, $|\dot u(t)|$ has local maxima at $t$ if and only
$t\in\partial I_i^+$, with $i\in L$. Lemma \ref{nehariprimeprop2} (points (i) and (ii)) also implies
properties (P1) and (P2) (that is, the asymptotic
of the solutions on $I^-_i$ and $I^+_i$ with $i \not\in L$).
Property (P3) (that is, the behavior on $I^+_i$ with $i \in L$) follows from the following claim:
\textit{for every $\eps > 0$ there exists $\delta > 0$ such that,
for any $u$ solving $\ddot u + a^+(t) u^3 = 0$ on $[0,\tau]$ with
$r^2 \leq \int_0^\tau \dot u^2 \leq 4c+\delta$
and $u(t) \geq 0$ for $t \in [\zeta,\tau-\zeta]$, the following implication holds:}
$$
\vert u(0) \vert, \vert u(\tau) \vert \leq \delta \quad \Longrightarrow \quad
\mathrm{dist}_{W^{2,\infty}}(u,\mathcal{K}_0) \leq \eps.
$$
Indeed, by Remark \ref{rem:property3}, we can assume that our solutions satisfy such conditions.

To prove the claim, assume by contradiction that there
is $\eps^* > 0$ and a sequence $(u_n)$ of solutions
with $r^2 \leq \int_0^\tau \dot u_n^2 \to 4c$,
$u_n(0), u_n(\tau) \to 0$ and
\begin{equation}\label{assurdopositivo}
\mathrm{dist}_{W^{2,\infty}}(u_n,\mathcal{K}_0) > \eps^*.
\end{equation}
By standard arguments, $u_n$ converges in $W^{2,\infty}$
to a limit $\bar u \in \mathcal{K}_0$, contradicting
\eqref{assurdopositivo}.

\medbreak
We now deal with the positivity, splitting our arguments into three steps.
\smallbreak
\noindent
\emph{\underline{Step 1.} It holds $u(\sigma_i) > 0$ and $u(\tau_i) > 0$
for every $i \in L$.}
\smallbreak
\noindent
To see this, we first recall that $\mathcal{K}_0$ is compact.
Hence, there exists a constant $m>0$, $m<\rho$, such that
$$
m \leq \dot u(\sigma^+) \quad \mbox{ and } \quad \dot u(\tau^-) \leq -m \quad  \quad \mbox{ for every } u \in \mathcal{K}_0,\;u>0.
$$
Hence, since (P3) holds true, for the solution found in Theorem \ref{thper} we have that
$m/2 \leq \dot u(\sigma^+_i)$ and $u(\tau^-_i) \leq -m/2$ for every $i \in L$.
Convexity arguments now imply the conclusion
($u(\sigma_i) \leq 0$ or $u(\tau_i) \leq 0$ would imply a contradiction with (P2)).
\smallbreak
\noindent
\emph{\underline{Step 2.} It holds $u(t) > 0$ for every $t \in I^+_i$ with $i \in L$.}
\smallbreak
\noindent
This follows from Lemma \ref{pochizeri}.
\smallbreak
\noindent
\emph{\underline{Step 3.} It holds $u(t) > 0$ for every $t \in I_N$.}
\smallbreak
\noindent
This follows from (iii) of Proposition \ref{propo:auxiliary}, since
(P1) and (P2) imply that \eqref{rotazioni} holds true.
\medbreak
We now go back to Theorem \ref{thmain}.
If $\mathcal{L} \in \Sigmadue$ is a given sequence, the solutions
$u = u_{\mathcal{L},N}$ constructed in Theorem \ref{thper} for
$N$ larger and larger of course are $L^{\infty}$-bounded,
independently on $N$. Then, the elementary inequality
$$
\Vert \dot v \Vert_{L^{\infty}(\mathbb{R})} \leq 2 \Vert v \Vert_{L^{\infty}(\mathbb{R})} +
\Vert \ddot v \Vert_{L^{\infty}(\mathbb{R})}, \quad \mbox{ for every } v \in W^{2,\infty}(\mathbb{R}),
$$
(compare with \eqref{diselementare}) implies that
$\Vert u_{\mathcal{L},N} \Vert_{W^{2,\infty}(\mathbb{R})}$ is bounded
independently on $N$.
Hence, Ascoli-Arzelà Theorem shows that, for
$N \to +\infty$,
$u_{\mathcal{L},N}$ converges, locally in $W^{2,\infty}$, to a limit function
$u_{\mathcal{L}} \in W^{2,\infty}(\mathbb{R})$.
It is easily seen that such a function satisfies all the requirements
of Theorem \ref{thmain}.
\medbreak
Finally, once the proof of Theorem \ref{thmain} is concluded, we have that Theorem \ref{thmain_intro}
follows immediately. As far as Theorem \ref{thmain_intro_conv} is concerned, the $W^{2,\infty}(I)$ convergence
$u_\mu$ to $\bar u\in\mathcal{K}_{\mathcal{L}}$ follows from properties (P2) and (P3) if $I=I^+_i$, and from Lemma \ref{decay} if $I=[\tau_i+\delta,\sigma_{i+1}-\delta]$, for any $[\tau_i,\sigma_{i+1}]=I^-_i$ and $\delta>0$. The $H^1_{\mathrm{loc}}$ convergence descends from the $W^{2,\infty}$ one, together with (P1). To prove the H\"older convergence, we first observe that (P1), (P2), (P3) also imply that $u_\mu\to\bar u$ uniformly. On the other hand, since for every $s\neq t$
\[
\text{both } \quad \frac{|u_\mu(t) - u_\mu(s)|}{|t-s|} \leq \rho \quad \text{ and }\quad
\frac{|\bar u(t) - \bar u(s)|}{|t-s|}\leq \rho
\]
(recall Remark \ref{rem:K_soddisfa_rho}), we have
\[
\frac{|(u_\mu-\bar u)(t) - (u_\mu-\bar u)(s)|}{|t-s|^\alpha} \leq (2\rho)^\alpha
|(u_\mu-\bar u)(t) - (u_\mu-\bar u)(s)|^{1-\alpha}\to 0
\]
by uniform convergence, concluding the proof.

\begin{remark}\label{rem:general_results}
As we mentioned, our method can be applied also in different situations. To conclude,
we briefly discuss some of them, with an emphasis on the minor changes they require
to be dealt with.

\textbf{Changing sign solutions.} As $\mathcal{K}_0$ is defined by even conditions, it
contains both positive and negative solutions of $\ddot u + a^+(t) u^3=0$ in $H^1_0(0,\tau)$.
As a consequence, if $\mathcal{L}\in\{-1,0,1\}^{\ZZ}$, one can redefine $\mathcal{K}_{\mathcal{L}}$
as
\[
\mathcal{K}_{\mathcal{L}} = \left\{u\in W^{1,\infty}(\RR) : u|_{I^+_i} \in\mathcal{K}_i,\,
\pm u|_{I^+_i}>0
\text{ if }\mathcal{L}_i=\pm1\text{ and }u\equiv0\text{ elsewhere}\right\},
\]
and obtain changing sign solutions by singularly perturbing such a set. In doing that, one has
to change the definition of $\mathcal{N}_\mu$ (equation \eqref{eq:defnehari}), by substituting
property (C2) with
\begin{itemize}
\item[(C2')] $\pm u(t) > 0$ for every $t \in [\sigma_i + \zeta, \tau_i - \zeta]$ with
$\mathcal{L}_i = \pm1$.
\end{itemize}
Consequently, one can find solutions having the appropriate sign on each $I_i^+$,
when $\mathcal{L}_i = \pm1$, and changing sign exactly once in every connected components
of the complementary set.

\textbf{Periodic weights with more than two nodal intervals.} If $a$ is periodic, but
changes sign more than once in $[0,T]$ (but still a finite number of times), then the main problems
arise from a notational point of view, while the conceptual one should be clear: to start with,
one has to partition $[0,T]$ in the juxtaposition of consecutive intervals $I_i^\pm$, according
to the rule that, for every $I^-_i=[\tau_i,\sigma_{i+1}]$ and $\delta>0$ small, it holds
\[
\int_{\tau_i}^{\tau_i+\delta}a^-(t)\,dt >0
\qquad \text{and} \qquad
\int^{\sigma_{i+1}}_{\sigma_{i+1}-\delta}a^-(t)\,dt >0
\]
(notice that this is always possible, exactly as in the one-zero case, up to carefully choosing the points $\sigma_i$, $\tau_i$).
Accordingly, $\mathcal{N}_\mu$ should be defined by means of different (still finite) constants
$r_i$, $\rho_i$, $\zeta_i$. As a byproduct, this should prove multiplicity of periodic solutions.

\textbf{Non periodic weights.} In the same spirit of what we have just enlightened, one
may also consider non necessarily periodic weights $a$, though enjoying some uniform oscillatory
properties (in particular, the constants $r_i$, $\rho_i$, $|I^\pm_i|$, $i\in\ZZ$,
should be bounded below and above independently on $i$; furthermore, a uniform version of Lemma
\ref{pochizeri} should hold). By considering periodic truncation of such a weight on larger
and larger intervals, one should eventually obtain bounded entire solutions, with possible
recursivity properties.

\textbf{Dirichlet and Neumann boundary conditions.} Beyond periodic ones,
other boundary conditions on bounded intervals $I$ can be considered in our construction.
Dirichlet homogeneous ones can be obtained by replacing $H^1_{\mathrm{per}}(I)$ with
$H^1_0(I)$, thus recovering the results by Gaudenzi, Habets and Zanolin
\cite{GauHabZan03,GauHabZan04}.
In the same way, using $H^1(I)$, one can solve the Neumann homogeneous problem on $I$.
From this point of view, the Neumann problem, as well as the periodic one, contains the
further difficulty that the quadratic part of the action functional does not
correspond to an equivalent norm in the ambient Hilbert space.
\end{remark}

\small
\subsection*{Acknowledgments}
Work partially supported  by the PRIN-2012-74FYK7 Grant:
``Variational and perturbative aspects of nonlinear differential problems''.
and by the project ERC Advanced Grant  2013 n. 339958:
``Complex Patterns for Strongly Interacting Dynamical Systems - COMPAT''.

\end{document}